\documentclass[10pt,reqno]{amsart}

\usepackage{amsmath,amssymb,mathrsfs,graphicx,enumitem,verbatim}

\usepackage[pdftex=true,
	bookmarks=true,
	bookmarksopen=true,
	bookmarksnumbered=true,
	pdftoolbar=true,
	pdfmenubar=true,
	pdffitwindow=false,
	pdfstartview={FitH},
	colorlinks=true,
	linkcolor=red,
	citecolor=blue,
	breaklinks=true]{hyperref}

\allowdisplaybreaks

\numberwithin{equation}{section}

\newtheorem{thm}{Theorem}[section]
\newtheorem{thmi}{Theorem}
\newtheorem{prop}[thm]{Proposition}
\newtheorem{lemma}[thm]{Lemma}

\newtheorem*{thm*}{Theorem}
\newtheorem*{prop*}{Proposition}
\newtheorem*{cor*}{Corollary}
\newtheorem*{conj*}{Conjecture}

\theoremstyle{definition}
\newtheorem{definition}[thm]{Definition}

\theoremstyle{remark}
\newtheorem{rmk}[thm]{Remark}

\newlength{\enummargin}
\newcommand{\mc}{\mathcal}

\newcommand{\cC}{\mc C}

\newcommand{\cE}{\mc E}
\newcommand{\cF}{\mc F}

\newcommand{\cH}{\mc H}

\newcommand{\cO}{\mc O}

\newcommand{\cT}{\mc T}

\newcommand{\cV}{\mc V}

\newcommand{\C}{\mathbb{C}}

\newcommand{\R}{\mathbb{R}}
\newcommand{\Z}{\mathbb{Z}}
\newcommand{\Sph}{\mathbb{S}}

\newcommand{\End}{\operatorname{End}}

\newcommand{\Hom}{\operatorname{Hom}}

\renewcommand{\Im}{\operatorname{Im}}
\renewcommand{\Re}{\operatorname{Re}}
\newcommand{\id}{\operatorname{id}}

\newcommand{\supp}{\operatorname{supp}}

\newcommand{\tr}{\operatorname{tr}}
\newcommand{\conn}[1]{\nabla^{#1}}
\newcommand{\pbconn}[1]{\wt\nabla^{#1}}
\newcommand{\chr}[1]{{}^{#1}\Gamma}

\newcommand{\la}{\langle}
\newcommand{\ra}{\rangle}
\newcommand{\pa}{\partial}
\newcommand{\tn}{\textnormal}
\newcommand{\ff}{\tn{ff}}
\newcommand{\eps}{\epsilon}

\newcommand{\wt}{\widetilde}
\newcommand{\wh}{\widehat}
\newcommand{\ol}{\overline}

\newcommand{\bl}{{\mathrm{b}}}
\newcommand{\cl}{{\mathrm{c}}}

\newcommand{\semi}{\hbar}

\newcommand{\Diff}{\mathrm{Diff}}
\newcommand{\Diffb}{\Diff_\bl}
\newcommand{\Vf}{\mathcal V}
\newcommand{\Vb}{\Vf_\bl}

\newcommand{\Psib}{\Psi_\bl}
\newcommand{\psdo}{ps.d.o.}
\newcommand{\Psih}{\Psi_\semi}

\newcommand{\Tb}{{}^{\bl}T}

\newcommand{\half}{\frac{1}{2}}

\newcommand{\ham}{H}
\newcommand{\sub}{{\mathrm{sub}}}

\newcommand{\CI}{\cC^\infty}
\newcommand{\CIdot}{\dot\cC^\infty}
\newcommand{\CIc}{\cC^\infty_\cl}

\newcommand{\Hb}{H_{\bl}}

\newcommand{\numin}{\nu_{\mathrm{min}}}

\newcommand{\openbigpmatrix}[1]{\addtolength{\arraycolsep}{-#1}\begin{pmatrix}}
\newcommand{\closebigpmatrix}[1]{\end{pmatrix}\addtolength{\arraycolsep}{#1}}

\begin{document}
\title[Tensor-valued waves on Kerr-de Sitter]{Resonance expansions for tensor-valued waves on asymptotically Kerr-de Sitter spaces}

\author{Peter Hintz}

\address{Department of Mathematics, Stanford University, CA 94305-2125, USA}
\email{phintz@math.stanford.edu}

\address{Department of Mathematics, University of California, Berkeley, CA 94720-3840, USA}
\email{phintz@berkeley.edu}

\date{March 16, 2015. Final revision: November 29, 2015}

\begin{abstract}
  In recent joint work with Andr\'as Vasy \cite{HintzVasyKdsFormResonances}, we analyze the low energy behavior of differential form-valued waves on black hole spacetimes. In order to deduce asymptotics and decay from this, one in addition needs high energy estimates for the wave operator acting on sections of the form bundle. The present paper provides these on perturbations of Schwarzschild-de Sitter spaces in all spacetime dimensions $n\geq 4$. In fact, we prove exponential decay, up to a finite-dimensional space of resonances, of waves valued in \emph{any} finite rank subbundle of the tensor bundle, which in particular includes differential forms and symmetric tensors. As the main technical tool for working on vector bundles that do not have a natural positive definite inner product, we introduce \emph{pseudodifferential inner products}, which are inner products depending on the position in phase space.
\end{abstract}

\maketitle

\section{Introduction}
\label{SecIntro}

We continue the analysis of (linear) aspects of the black hole stability problem by studying linear \emph{tensor-valued} wave equations on perturbations of Schwarzschild-de Sitter spaces with spacetime dimension $n\geq 4$; in particular, this includes wave equations for differential forms and symmetric 2-tensors. In our main result, we establish exponential decay up to a finite-dimensional space of resonances:

\begin{thmi}
\label{ThmIntroWaveExpansion}
  Let $(M,g)$ denote a Kerr-de Sitter spacetime with small angular momentum. Let $\cE\subset\cT_k$ be a subbundle of the bundle $\cT_k$ of (covariant) rank $k$ tensors on $M$, so that the tensor wave operator $\Box_g=-\tr\nabla^2$ acts on sections of $\cE$; for instance, one can take $\cE$ to be equal to $\cT_k$, symmetric rank $k$-tensors or differential forms of degree $k$. Let $\Omega$ denote a small neighborhood of the domain of outer communications, bounded beyond but close to the cosmological and the black hole horizons by spacelike boundaries, and let $t_*$ be a smooth time coordinate on $\Omega$. See Figure~\ref{FigDomain} for the setup.

  Then for any $f\in\CIc(\Omega,\cE)$, the wave equation $\Box_g u=f$ has a unique global forward solution (supported in the causal future of $\supp f$) $u\in\CI(\Omega,\cE)$, and $u$ has an asymptotic expansion
  \[
    u = \sum_{j=1}^N\sum_{m=0}^{m_j-1}\sum_{\ell=1}^{d_j} e^{-it_*\sigma_j}t_*^m u_{jm\ell}a_{jm\ell}(x) + u',
  \]
  where $u_{jm\ell}\in\C$, the resonant states $a_{jm\ell}$, only depending on $\Box_g$, are smooth functions of the spatial coordinates and $\sigma_j\in\C$ are resonances with $\Im\sigma_j>-\delta$ (whose multiplicity is $m_j\geq 1$ and for which the space of resonant states has dimension $d_j$), while $u'\in e^{-\delta t_*}L^\infty(\Omega,\cE)$ is exponentially decaying, for $\delta>0$ small; we measure the size of sections of $\cE$ by means of a $t_*$-independent positive definite inner product.
\end{thmi}

The same result holds true if we add any stationary $0$-th order term to $\Box$, and one can also add stationary first order terms which are either small or subject to a natural, but somewhat technical condition, which we explain in Remark~\ref{RmkFirstOrderPerturbation}. In fact, we can even work on spacetimes which merely \emph{approach a stationary perturbation of Schwarzschild-de Sitter space of any spacetime dimension $n\geq 4$ exponentially fast}. See \S\ref{SecProof} for the form of the Schwarzschild-de Sitter metric and the precise assumptions on regularity and asymptotics of perturbations, for details on the setup, and Theorem~\ref{ThmWaveExpansion} for the full statement of Theorem~\ref{ThmIntroWaveExpansion}.

\begin{figure}[!ht]
\centering
\includegraphics{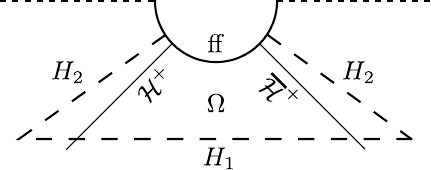}
\caption{Setup for Theorem~\ref{ThmIntroWaveExpansion} and Theorem~\ref{ThmWaveExpansion} below. Shown are the black hole horizon $\cH^+$ and the cosmological horizon $\overline{\cH}^+$, beyond which we put an artificial spacelike hypersurface $H_2$ with two connected components. The hypersurface $H_1$ plays the role of a Cauchy hypersurface, and the forcing as well as the solution to the wave equation are supported in its causal future. The domain $\Omega$ is bounded by the hypersurfaces $H_1$ and $H_2$. The `point at future infinity' in the usual Penrose diagrammatic representation is shown blown-up here, since the wave operator is well-behaved (namely, a b-operator in the sense of Melrose \cite{MelroseAPS}) on the blown-up space, and the asymptotic information is encoded on the front face $\ff$ of the blow-up.}
\label{FigDomain}
\end{figure}

The resonances and resonant states depend strongly on the precise form of the operator and which bundle one is working on. In the case of the trivial bundle, thus considering scalar waves, they were computed in the Kerr-de Sitter setting by Dyatlov \cite{DyatlovQNM}, following work by S\'a Barreto and Zworski \cite{SaBarretoZworskiResonances} as well as Bony and H\"afner \cite{BonyHaefnerDecay}. In recent work with Vasy \cite{HintzVasyKdsFormResonances}, we compute the resonances for the Hodge d'Alembertian on differential forms, which equals the tensor wave operator plus a zeroth order curvature term: We show that there is only one resonance $\sigma_1=0$ in $\Im\sigma\geq 0$, with multiplicity $m_1=1$, and we canonically identify the $0$-resonant states with cohomological information of the underlying spacetime. Note however that \cite{HintzVasyKdsFormResonances} deals with a very general class of warped product type spacetimes with asymptotically hyperbolic ends, while the present paper is only concerned with (perturbations of) Schwarzschild-de Sitter spacetimes. We remark that in general one expects that $\Box_g=-\tr\nabla^2$ on a bundle $\cE$ as in Theorem~\ref{ThmIntroWaveExpansion} has resonances in $\Im\sigma>0$, thus causing linear waves to grow exponentially in time.

We point out that if there are no resonances for $\Box_g$ (plus lower order terms) in $\Im\sigma\geq 0$, thus solutions decay exponentially, we can combine Theorem~\ref{ThmIntroWaveExpansion} with the framework for quasilinear wave-type equations developed by the author \cite{HintzQuasilinearDS} and in collaboration with Vasy \cite{HintzVasyQuasilinearKdS} and immediately obtain the \emph{global solvability of quasilinear equations}. This also works if there is merely a simple resonance at $\sigma=0$ which is annihilated by the nonlinearity.

The specific point of view from which we approach the proof of Theorem~\ref{ThmIntroWaveExpansion} was developed by Vasy~\cite{VasyMicroKerrdS} and extended by Vasy and the author \cite{HintzVasySemilinear}, building on a number of earlier works. In the context of scalar waves, more general and precise versions of Theorem~\ref{ThmIntroWaveExpansion} are known. See the references below. Thus, the main novelty is that we give a conceptually transparent framework that allows us to deal with tensor-valued waves on black hole spacetimes, where the natural inner product on the tensor bundle induced by the spacetime metric is not positive definite. The central motivation for the study of such waves is the \emph{black hole stability problem}, see the lecture notes by Dafermos and Rodnianski \cite{DafermosRodnianskiLectureNotes} for details. Notice that in order to obtain energy estimates for waves, one needs to work with positive inner products on the tensor bundle, relative to which however $\Box$ is in general not well-behaved: Most severely, it is in general far from being symmetric at the trapped set, which prevents the use of estimates at normally hyperbolic trapping. In the context of black hole spacetimes, such estimates were pioneered by Wunsch and Zworski \cite{WunschZworskiNormHypResolvent} and Dyatlov \cite{DyatlovResonanceProjectors,DyatlovSpectralGaps}. On a pragmatic level, we show that one can conjugate $\Box$ by a suitable $0$-th order pseudodifferential operator so as to make the conjugated operator (almost) symmetric at the trapped set with respect to a positive definite inner product, and one can then directly apply Dyatlov's results \cite{DyatlovSpectralGaps}. In other words, we reduce the high frequency analysis of tensor-valued waves to an essentially scalar problem. The conceptually correct point of view to accomplish this conjugation is that of \emph{pseudodifferential inner products}, which we introduce in this paper.

Roughly speaking, pseudodifferential inner products replace ordinary inner products $\int\la B_0(u),v\ra\,|dg|$, where $B_0$ is an inner product on the fibers of $\cE$, mapping $\cE$ into its anti-dual $\ol{\cE}^*$, by `inner products' of the form $\int \la B(x,D)u,v\ra\,|dg|$, where $B\in\Psi^0$ is a zeroth order pseudodifferential operator mapping sections of $\cE$ into sections of $\ol{\cE}^*$. Thus, we gain a significant amount of flexibility, since we can allow the inner product to \emph{depend on the position in phase space}, rather than merely on the position in the base: Indeed, the principal symbol $b=\sigma_0(B)$ is an inner product on the vector bundle $\pi^*\cE$ over $T^*M\setminus 0$, where $\pi\colon T^*M\setminus 0\to M$ is the projection.

One can define adjoints of operators $P\in\Psi^m(M,\cE)$ (e.g.\ $P=\Box_g$), acting on sections of $\cE$, relative to a pseudodifferential inner product $B$, denoted $P^{*B}$, which are well-defined modulo smoothing operators. Moreover, there is an invariant symbolic calculus involving the \emph{subprincipal operator} $S_\sub(P)$, which is a first order differential operator on $T^*M\setminus 0$ acting on sections of $\pi^*\cE$ that invariantly encodes the subprincipal part of $P$, for computing principal symbols of commutators and imaginary parts of such operators. In the case that $P$ is principally scalar and real, the principal symbol of $P-P^{*B}\in\Psi^{m-1}(M,\cE)$ then vanishes in some conic subset of phase space $T^*M\setminus 0$ if and only if $S_\sub(P)-S_\sub(P)^{*b}$ (taking the adjoint with respect to the inner product $b$) does, which in turn can be reinterpreted as saying that the principal symbol of $QPQ^{-1}-(QPQ^{-1})^{*B_0}$ vanishes there, where $B_0$ is an ordinary inner product on $\cE$, and $Q\in\Psi^0(M,\cE)$ is a suitably chosen elliptic operator. In the case considered in Theorem~\ref{ThmIntroWaveExpansion} then, it turns out that the subprincipal operator of $\Box_g$ on tensors, decomposed into parts acting on tangential and normal tensors according to the product decompositions $M=\R_t\times X_x$ and $X=(r_-,r_+) \times \Sph^{n-2}$, at the trapped set equals the derivative along the Hamilton vector field $\ham_G$, $G$ the dual metric function, plus a \emph{nilpotent} zeroth order term. This then enables one to choose a positive definite inner product $b$ on $\pi^*\cE$ relative to which $S_\sub(\Box_g)$ is arbitrarily close to being symmetric at the trapped set; thus with $B=b(x,D)$, the operator $\Box_g$ is arbitrarily close to being symmetric with respect to the pseudodifferential inner product $B$. Hence, one can indeed appeal to Dyatlov's results on spectral gaps by considering a conjugate of $\Box_g$, which is the central ingredient in the proof of Theorem~\ref{ThmIntroWaveExpansion}.

We point out that refined microlocal propagation results, in the sense of polarization sets, for systems of real principal type were proved by Dencker \cite{DenckerPolarization}, and in fact the subprincipal operator we define here is very closely related to the partial connection along the Hamilton flow defined in \cite{DenckerPolarization}; see Remark~\ref{RmkDencker} for details. In fact, for principally scalar operators, which are the focus of the present paper, the partial connection of \cite{DenckerPolarization} is canonically defined --- not merely up to a rescaling --- and agrees (up to a factor of $i$) with the subprincipal operator; from this perspective, the present paper shows that Dencker's partial connection turns out to play a key role also for a certain kind of quantitative analysis of (subprincipally) non-scalar operators.

\subsection{Related work}

The study of non-scalar waves on black hole backgrounds has focused primarily on Maxwell's equations: Sterbenz and Tataru \cite{SterbenzTataruMaxwellSchwarzschild} showed local energy decay for Maxwell's equations on a class of spherically symmetric asymptotically flat spacetimes including Schwarzschild. Blue \cite{BlueMaxwellSchwarzschild} established conformal energy and pointwise decay estimates in the exterior of the Schwarzschild black hole; Andersson and Blue \cite{AnderssonBlueMaxwellKerr} proved similar estimates on slowly rotating Kerr spacetimes. These followed earlier results for Schwarzschild by Inglese and Nicolo \cite{IngleseNicoloMaxwellSchwarzschild} on energy and pointwise bounds for integer spin fields in the far exterior of the Schwarzschild black hole, and by Bachelot \cite{BachelotSchwarzschildScattering}, who proved scattering for electromagnetic perturbations. Finster, Kamran, Smoller and Yau \cite{FinsterKamranSmollerYauDiracKerr} proved local pointwise decay for Dirac waves on Kerr. There are further works which in particular establish bounds for certain components of the Maxwell field, see Donninger, Schlag and Soffer \cite{DonningerSchlagSofferSchwarzschild} and Whiting \cite{WhitingKerrModeStability}. Dafermos \cite{DafermosEinsteinMaxwellScalarStability,DafermosBlackHoleNoSingularities} studied the non-linear Einstein-Maxwell-scalar field system under the assumption of spherical symmetry.

The framework in which we describe resonances was introduced by Vasy \cite{VasyMicroKerrdS}. In the scalar setting, this can directly be combined with estimates at normally hyperbolic trapping by Dyatlov \cite{DyatlovResonanceProjectors,DyatlovSpectralGaps} and Nonnenmacher and Zworski \cite{NonnenmacherZworskiDecay}, building on \cite{WunschZworskiNormHypResolvent}, to obtain resonance expansions for scalar waves. On exact Kerr-de Sitter space, Dyatlov proved a significant strengthening of this in \cite{DyatlovAsymptoticDistribution}, obtaining a full resonance expansion for scalar waves, improving on the result of Bony and H\"afner \cite{BonyHaefnerDecay} and Melrose, S\'a Barreto and Vasy \cite{MelroseSaBarretoVasySdS} in the Schwarzschild-de Sitter setting, which in turn followed S\'a Barreto and Zworski \cite{SaBarretoZworskiResonances}. Vasy \cite{VasyHyperbolicFormResolvent} proved the meromorphic continuation of the resolvent of the Laplacian on differential forms on asymptotically hyperbolic spaces (following earlier works by Mazzeo and Melrose \cite{MazzeoMelroseHyp} and Guillarmou \cite{GuillarmouMeromorphic} in the scalar setting and Mazzeo \cite{MazzeoHodgeCohomology}, Carron and Pedon \cite{CarronPedonForms} and Guillarmou, Moroianu and Park \cite{GuillarmouMoroianuParkDirac} for forms and spinors; see also the work of Dyatlov, Faure and Guillarmou \cite{DyatlovFaureGuillarmouPowerSpectrum}, which in particular involves a discussion of Laplacians on compact hyperbolic manifolds acting on symmetric tensors). The fact that the analysis presented in \cite{VasyMicroKerrdS}, which underlies \cite{VasyHyperbolicFormResolvent}, works on sections of vector bundles just as it does on functions is fundamental for the present paper.

There is a large literature on linear \emph{scalar} waves on black hole spacetimes, see the works by Dafermos, Rodnianski \cite{DafermosRodnianskiSdS,DafermosRodnianskiKerrBoundedness,DafermosRodnianskiKerrDecaySmall} and Dafermos, Rodnianski and Shlapentokh-Rothman \cite{DafermosRodnianskiShlapentokhRothmanDecay}, following work by Wald \cite{WaldSchwarzschild} and Kay and Wald \cite{KayWaldSchwarzschild}, further Tataru \cite{TataruDecayAsympFlat} as well as Tataru and Tohaneanu \cite{TataruTohaneanuKerrLocalEnergy}; further references are given in the introduction of \cite{VasyMicroKerrdS}.

\subsection{Structure of the paper}

In \S\ref{SecProof}, we recall the Schwarzschild-de Sitter metric and its extension past the horizons, put it into the framework of \cite{HintzVasySemilinear,VasyMicroKerrdS} for the study of asymptotics of waves, and establish the normally hyperbolic nature of its trapping. We proceed to sketch the proof of Theorem~\ref{ThmIntroWaveExpansion}, leaving the discussion of high energy estimates at the trapped set to the subsequent sections, which comprise the central part of the paper: We introduce pseudodifferential inner products on vector bundles in full generality in \S\ref{SecPsdoInner}, and we use the theory developed there in \S\ref{SecSubprincipalLaplace} to study pseudodifferential inner products for wave operators on tensor bundles, uncovering the nilpotent nature of the subprincipal operator of $\Box$ on Schwarzschild-de Sitter space at the trapping in \S\ref{SubsecSDS} and thereby finishing the proof of Theorem~\ref{ThmIntroWaveExpansion}.

\subsection*{Acknowledgments}

I am very grateful to Andr\'as Vasy and Alexandr Zamorzaev for many useful discussions, and to Semyon Dyatlov for suggesting the definition of the subprincipal operator in a closely related context. I am also grateful to an anonymous referee for many helpful comments and suggestions.

I gratefully acknowledge support by a Gerhard Casper Stanford Graduate Fellowship and Andr\'as Vasy's National Science Foundation grants DMS-1068742 and DMS-1361432.

\section{Detailed setup and proof of the main theorem}
\label{SecProof}

We recall the form of the $n$-dimensional Schwarzschild-de Sitter metric, $n\geq 4$: We equip $M=\R_t\times X$, $X=(r_-,r_+)_r\times\Sph^{n-2}_\omega$, with $r_\pm$ defined below, with the metric
\begin{equation}
\label{EqSDSMetric}
  g_0 = \mu\,dt^2-(\mu^{-1}\,dr^2+r^2\,d\omega^2),
\end{equation}
where $d\omega^2$ is the round metric on the sphere $\Sph^{n-2}$, and $\mu=1-\frac{2M_\bullet}{r^{n-3}}-\lambda r^2$, $\lambda=\frac{2\Lambda}{(n-2)(n-1)}$, with $M_\bullet>0$ the black hole mass and $\Lambda>0$ the cosmological constant. The assumption
\begin{equation}
\label{EqSDSNondegeneracy}
  M_\bullet^2\lambda^{n-3}<\frac{(n-3)^{n-3}}{(n-1)^{n-1}}
\end{equation}
guarantees that $\mu$ has two unique positive roots $0<r_-<r_+$. Indeed, let $\wt\mu=r^{-2}\mu=r^{-2}-2M_\bullet r^{1-n}-\lambda$. Then $\wt\mu' = -2r^{-n}(r^{n-3}-(n-1)M_\bullet)$ has a unique positive root $r_p=[(n-1)M_\bullet]^{1/(n-3)}$, $\wt\mu'(r)>0$ for $r\in(0,r_p)$ and $\wt\mu'(r)<0$ for $r>r_p$; moreover, $\wt\mu(r)<0$ for $r>0$ small and $\wt\mu(r)\to-\lambda<0$ as $r\to\infty$, thus the existence of the roots $0<r_-<r_+$ of $\wt\mu$ is equivalent to the requirement $\wt\mu(r_p)=\frac{n-3}{n-1}r_p^{-2}-\lambda>0$, which is equivalent to \eqref{EqSDSNondegeneracy}.

Define $\alpha=\mu^{1/2}$, thus $d\alpha=\frac{1}{2}\mu'\alpha^{-1}\,dr$, and let
\begin{equation}
\label{EqBetaPM}
  \beta_\pm(r):=\mp\frac{2}{\mu'(r)}
\end{equation}
near $r_\pm$, so $\beta_\pm(r_\pm)>0$ there. Then the metric $g_0$ can be written as
\[
  g_0 = \alpha^2\,dt^2-h,\quad h=\alpha^{-2}\,dr^2+r^2\,d\omega^2=\beta_\pm^2\,d\alpha^2+r^2\,d\omega^2,
\]
We introduce a new time variable $t_*=t-F(\alpha)$, with $\pa_\alpha F=-\alpha^{-1}\beta_\pm$ near $r=r_\pm$. Then
\[
  g_0 = \mu\,dt_*^2 - \beta_\pm\,dt_*\,d\mu - r^2\,d\omega^2
\]
near $r=r_\pm$, which extends as a non-degenerate Lorentzian metric to a neighborhood $\wt M=\R_{t_*}\times\wt X$ of $M$, where $\wt X=(r_--2\delta,r_++2\delta)\times\Sph^{n-2}$. We will consider the Cauchy problem for the tensor wave equation in the domain $\Omega\subset\wt M$,
\[
  \Omega = [0,\infty)_{t_*} \times [r_--\delta,r_++\delta]_r \times \Sph^{n-2}.
\]
Thus, $\Omega$ is bounded by the Cauchy surface $H_1=\{t_*=0\}$, which is spacelike, and by the hypersurface $H_2=\bigcup_\pm \{r=r_\pm \pm \delta\}$, which has two spacelike components, one lying beyond the black hole ($r_-$) and the other beyond the cosmological ($r_+$) horizon; see Figure~\ref{FigDomain}.

For the purpose of analysis on spacetimes close to (but not necessarily asymptotically equal to!) Schwarzschild-de Sitter space, we encode the uniform (asymptotically stationary) structure of the spacetime by working on a compactified model, which puts the problem into the setting of Melrose's \emph{b-analysis}, see \cite{MelroseAPS}: Define $\tau:=e^{-t_*}$, and partially compactify $\wt M$ to a manifold $\ol M$ with boundary by adding $\tau=0$ as the boundary at future infinity and declaring $\tau$ to be a smooth boundary defining function. The metric $g_0$ becomes a smooth \emph{Lorentzian b-metric} on $\ol M$: If $dx_i$ denotes coordinate differentials on $\wt X$, then $g_0$ is a linear combination of $\frac{d\tau^2}{\tau^2}$, $\frac{d\tau}{\tau}\otimes dx_i+dx_i\otimes\frac{d\tau}{\tau}$ and $dx_i\otimes dx_j$ with coefficients which are smooth on $\ol M$, and $g_0$, written in such coordinates, is a non-degenerate matrix (with Lorentzian signature) \emph{up to and including} $\tau=0$. Invariantly, we have the Lie algebra $\Vb(\ol M)$ of \emph{b-vector fields}, which are the vector fields tangent to the boundary, spanned by $\tau\pa_\tau=-\pa_{t_*}$ and $\pa_{x_i}$; elements of $\Vb(\ol M)$ are sections of a natural vector bundle $\Tb\ol M$, the b-tangent bundle, and we have the dual bundle $\Tb^*\ol M$, spanned by $\frac{d\tau}{\tau}$ and $dx_i$. Thus, $g$ is a smooth non-degenerate section of the symmetric second tensor power $S^2\Tb^*\ol M$.

Now, given a complex vector bundle $\cE\to\ol M$ of finite rank, equip it with an arbitrary Hermitian inner product and any smooth b-connection, which gives a notion of differentiating sections of $\cE$ along b-vector fields; over $\Omega$ (which has compact closure in $\ol M$), all choices of inner products are equivalent. We can then define the b-Sobolev space $\Hb^s(\Omega,\cE)$ for $s\in\Z_{\geq 0}$ to consist of all sections of $\cE$ over $\Omega$ which are square integrable (with respect to the volume density $|dg|$ induced by the metric $g$) together with all of its b-derivatives up to order $s$, and extend this to all $s\in\R$ by duality and interpolation, or via the use of b-pseudodifferential operators. For the forward problem for the wave equation, we work on spaces of functions which vanish in the past of $H_1$ and which extend across $H_2$. Thus, we work with the space $\Hb^s(\Omega,\cE)^{\bullet,-}$ of distributions $u\in\Hb^s(\Omega,\cE)$ which are extendible distributions at $H_2$ and supported distributions at $H_1$, i.e.\ they are restrictions to $\Omega$ of distributions on $\ol M$ which are supported in $t_*\geq 0$. See H\"ormander \cite[Appendix~B]{HormanderAnalysisPDE3} for details. We also have weighted b-Sobolev spaces $\Hb^{s,r}(\Omega,\cE)=\tau^r\Hb^s(\Omega,\cE)$, likewise for spaces of supported/extendible distributions. Note that the b-Sobolev spaces $\Hb^s$ are independent of the choice of boundary defining function $\tau$ in that the choice $\tau'=a\tau^\gamma$, $a=a(x)$ smooth, $\gamma>0$, while it changes the smooth structure of $\ol M$, yields the same spaces $\Hb^s$ with equivalent norms. The asymptotic behavior of waves will be encoded on the boundary $\pa_\infty\Omega$ at future infinity of $\Omega$, that is, on
\[
  \pa_\infty\Omega=\{\tau=0\}\times[r_--\delta,r_++\delta]_r\times\Sph^{n-2},
\]
which is a smooth manifold with boundary. Similarly to the above definitions, we can define Sobolev spaces (including semiclassical versions of these) with supported/extendible character at the boundary.

Suppose $g$ is a Lorentzian b-metric such that for some smooth Lorentzian b-metric $g'$, we have $g-g'\in\Hb^{\infty,r}(\Omega,S^2\Tb^*\ol M)$ for some $r>0$. (By the discussion of b-Sobolev spaces above, this condition on $g$ is invariant, i.e.\ independent of the specific choice of the boundary defining function $e^{-t_*}$ of the spacetime at future infinity.) Changing $g'$ so as to make it invariant under time translations does not affect this condition, since the difference between $g'$ and the metric obtained from $g'$ by replacing the metric coefficients (which are smooth on $\ol M$) by their values at $\tau=0$ lies in $\tau\CI(\Omega,S^2\Tb^*\ol M)\subset\Hb^{\infty,1}(\Omega,S^2\Tb^*\ol M)$; thus, let us assume $g'$ is $t_*$-invariant, i.e.\ its coefficients are independent of $t_*$ (equivalently, $\tau$). We will consider the wave operator $\Box_g$ acting on sections of the bundle $\cT_k$ of covariant tensors of rank $k$ over $\Omega$. We assume that $g'$ and $g_0$ are close (in the $C^k$ sense for sufficiently large $k$), so that the dynamical and geometric structure of $g$ is close to that of $g_0$ (see \cite[\S{3}]{HintzVasySemilinear} and \cite[\S{5}]{HintzVasyQuasilinearKdS} for details); in other words, the metric $g$ is exponentially approaching a stationary metric close to the Schwarzschild-de Sitter metric, so for instance perturbations (within this setting) of Kerr-de Sitter spaces are allowed. Most importantly, we require that the nature of the trapping for $g'$ (and thus for $g$) still be normally hyperbolic, and the subprincipal operator (see \S\ref{SubsecInvariantSubprincipal}) of $\Box_g$ at the trapped set, while not necessarily having the nilpotent structure alluded to in the introduction and explained in \S\ref{SubsecSDS}, have small imaginary part relative to (the symbol of) a pseudodifferential inner product on $\cT_k$. We point out that we will show the $r$-normal hyperbolicity for every $r$ of the trapping for Schwarzschild-de Sitter space in all spacetime dimensions below, and $r$-normal hyperbolicity (for large, but finite $r$) is structurally stable under perturbations of the metric, see Dyatlov~\cite{DyatlovResonanceProjectors} and Hirsch, Shub and Pugh \cite{HirschShubPughInvariantManifolds}. We then have:

\begin{thm}
\label{ThmWaveExpansion}
  In the above notation, if $g'$ is sufficiently close to the Schwarzschild-de Sitter metric $g_0$, then there exist $s_0\in\R$ and $\delta>0$ as well as a finite set $\{\sigma_j\colon j=1,\ldots,N\}\subset\C$, $\Im\sigma_j>-\delta$, integers $m_j\geq 1$ and $d_j\geq 1$, and smooth functions $a_{jm\ell}\in\CI(\pa_\infty\Omega)$, $1\leq j\leq N$, $0\leq m\leq m_j-1$, $1\leq\ell\leq d_j$, such that the following holds: The equation
  \begin{equation}
  \label{EqTensorWave}
    \Box_g u=f, \quad f \in \Hb^{s,\delta}(\Omega,\cT_k)^{\bullet,-}, \quad s\geq s_0,
  \end{equation}
  has a unique solution $u\in\Hb^{-\infty,-\infty}(\Omega,\cT_k)^{\bullet,-}$, which has an asymptotic expansion
  \[
    u = \chi(\tau)\sum_{j=1}^N \sum_{m=0}^{m_j-1} \sum_{\ell=1}^{d_j} \tau^{i\sigma_j}|\log\tau|^m u_{jm\ell}a_{jm\ell} + u',
  \]
  where $\chi$ is a cutoff function, i.e.\ $\chi(\tau)\equiv 1$ near $\tau=0$ and $\chi(\tau)\equiv 0$ near the Cauchy surface $H_1$, and $u_{jm\ell}\in\C$, while the remainder term is $u'\in\Hb^{s,\delta}(\Omega,\cT_k)^{\bullet,-}$.

  The same result holds true if we restrict to a subbundle of $\cT_k$ which is preserved by the action of $\Box$, for instance the degree $k$ form bundle, or the symmetric rank $k$ tensor bundle.

  If $V\in\CI(\ol M,\End(\cT_k))+\Hb^{\infty,r}(\Omega,\End(\cT_k))$, $r>0$, is a smooth (conormal) $\End(\cT_k)$-valued potential (without restriction on its size), the analogous result holds for $\Box_g$ replaced by $\Box_g+V$. We may even change $\Box_g$ by adding a first order b-differential operator $L$ acting on $\cT_k$ with coefficients which are elements of $\CI+\Hb^{\infty,r}$, provided either the coefficients of $L$ are small, or the subprincipal operator of $\Box_g+L$ is sufficiently close to being symmetric with respect to a pseudodifferential inner product on $\cT_k$, see Remark~\ref{RmkFirstOrderPerturbation}.
\end{thm}

The numbers $\sigma_j$ are called \emph{resonances} or \emph{quasinormal modes}, and the functions $a_{jm\ell}$ \emph{resonant states}. They have been computed in various special cases; see the discussion in the introduction for references. The threshold regularity $s_0$ is related to the dynamics of the flow of the Hamiltonian vector field $\ham_G$ of the dual metric function $G$ (i.e.\ $G(x,\xi)=|\xi|_{G(x)}^2$, with $G$ the dual metric of $g$) near the horizons which are generalized \emph{radial sets}, see \cite[Proposition~2.1]{HintzVasySemilinear}. Thus, $s_0$ can easily be made explicit, but this is not the point of the present paper.

The proof of Theorem~\ref{ThmWaveExpansion} proceeds in the same way as the proof of \cite[Theorem~2.20]{HintzVasySemilinear} in the scalar setting, so we shall be brief: Denote by $N(\Box_g)$ the normal operator of $\Box_g$: We freeze the coefficients of $\Box_g\in\Diffb^2(\Omega,\cT_k)$ at $\pa_\infty\Omega$ and thus obtain a dilation-invariant operator $N(\Box_g)$, with $\Box_g-N(\Box_g)$ being an operator whose coefficients decay exponentially (in $t_*$) by assumption on the structure of $g$. Denote by $\wh{\Box_g}(\sigma)\in\Diff^2(\pa_\infty\Omega,\cT_k)$ the Mellin transformed normal operator family, depending holomorphically on $\sigma\in\C$, which we obtain from $N(\Box_g)$ by replacing $D_{t_*}$ by $-\sigma$. (Note that changing the boundary defining function $\tau$ to $a(x)\tau^\gamma$, we can express the normal operator with respect to the new defining function in terms of the normal operator with respect to $\tau$, namely it equals $a(x)^{-1}\wh{\Box_g}(\gamma\sigma)a(x)$.) Once we show \emph{high energy estimates} for $\wh{\Box_g}(\sigma)^{-1}$, which are polynomial bounds on its operator norm between suitable Sobolev spaces as $|\Re\sigma|\to\infty$ in $\Im\sigma>-\delta$, we can use a contour shifting argument to iteratively improve on the decay of $u$, picking up contributions of the poles of $\wh{\Box_g}(\sigma)^{-1}$ which give rise to the resonance expansion. Concretely, these spaces are semiclassical Sobolev spaces with extendible character at the boundary of $\pa_\infty\Omega$, see in particular \cite{VasyMicroKerrdS} and the proof of \cite[Theorem~2.20]{HintzVasySemilinear}. Furthermore, as shown by Vasy \cite[\S{7}]{VasyMicroKerrdS}, these high energy estimates in $\Im\sigma\gg 0$ are automatic if the boundary defining function of future infinity is timelike; our choice does not satisfy this, but changing $t_*$ by a smooth function of the spatial variables, this can easily be arranged, see \cite[\S{6}]{VasyMicroKerrdS}, and in fact we can arrange $t_*=t$ away from the black hole and cosmological horizons. The fact that the remainder term $u'$ has the same regularity as the forcing term $f$, thus $u'$ loses $2$ derivatives relative to the elliptic gain of $2$ derivatives, comes from the high energy estimate losing a power of $2$, see \cite[Theorem~5.5]{HintzVasyQuasilinearKdS}, which in turn is caused by the same loss for high energy estimates at normally hyperbolic trapping, see \cite[Theorem~1]{DyatlovSpectralGaps}, or \cite[Theorem~4.5]{HintzVasyQuasilinearKdS} for a microlocalized version of Dyatlov's estimate.

Thus, the crucial point is to obtain high energy estimates at the trapped set for the operator $\Box_g$ acting on $\cT_k$ in $\Im\sigma>-\delta$. Dyatlov's result \cite[Theorem~1]{DyatlovSpectralGaps} (see also the discussion preceding \cite[Theorem~5.5]{HintzVasyQuasilinearKdS}) shows that a sufficient condition for these to hold is
\begin{equation}
\label{EqBoundAtTrapping}
  |\sigma|^{-1}\sigma_{\bl,1}\Bigl(\frac{1}{2i}(\Box_g-\Box_g^*)\Bigr)<\numin/2
\end{equation}
at the trapped set $\Gamma$, where $\numin$ is the minimal normal expansion rate of the Hamilton flow at the trapping, see \cite{DyatlovSpectralGaps} and the computation below. (We work in the b-setting here, which via the Mellin transform is equivalent, at least on the normal operator level, which is all that matters, to the semiclassical setting considered in Dyatlov's work; see the discussion in \cite[\S{5}]{HintzVasyQuasilinearKdS}.) Here, the adjoint is taken with respect to a \emph{positive definite inner product} on $\cT_k$; note that the inner product induced by $g$, with respect to which $\Box_g$ is of course symmetric, is not positive definite, except when $k=0$, i.e.\ for the scalar wave equation. Since $g$ is close to the Schwarzschild-de Sitter metric, it suffices to obtain such a bound for the Schwarzschild-de Sitter metric $g_0$. While this bound is impossible to obtain directly for the full range of Schwarzschild-de Sitter spacetimes, we show in \S\ref{SubsecSDS} how it can be obtained if we use pseudodifferential products, see Definition~\ref{DefPsdoInnerProduct}. We refer to Proposition~\ref{PropImagSymbolInv} for the efficient calculation of $\Box_g-\Box_g^*$ in an abstract setting, with the adjoint taken relative to a pseudodifferential inner product --- see Definition~\ref{DefPsdoAdjoint} --- and using the subprincipal operator defined in Definition~\ref{DefInvSubpr}; for $\Box_g$ concretely, the subprincipal operator is computed in Propositions~\ref{PropLaplaceSubpr} and \ref{PropSDSSubpr}. Prosaically, as we show in Proposition~\ref{PropPsdoInnerAsConjugation}, the use of a pseudodifferential inner product is equivalent to considering a conjugated operator $P:=Q\Box_g Q^-$, where $Q\in\Psib^0(\ol M,\cT_k)$ is a carefully chosen elliptic operator with parametrix $Q^-$: For any $\eps>0$, we can arrange $|\sigma|^{-1}\sigma_{\bl,1}(\frac{1}{2i}(P-P^*))<\eps$ (with the adjoint taken relative to an ordinary positive definite inner product on $\cT_k$), thus \eqref{EqBoundAtTrapping} holds for $\Box_g$ replaced by $P$; we will prove this in Theorem~\ref{ThmSDSNilpotentAtTrapping}. Hence \cite[Theorem~1]{DyatlovSpectralGaps} applies to $P$, establishing a spectral gap; indeed, by the remark following \cite[Theorem~1]{DyatlovSpectralGaps}, Dyatlov's result applies for operators on bundles as well, \emph{as soon as one establishes \eqref{EqBoundAtTrapping}}. Arranging \eqref{EqBoundAtTrapping} in a natural fashion lies at the heart of \S\S\ref{SecPsdoInner} and \ref{SecSubprincipalLaplace}.

It remains to establish the $r$-normal hyperbolicity for all $r$ for the Schwarzschild-de Sitter metric. The dynamics at the trapping only depend on properties of the (scalar!) principal symbol of $\Box_{g_0}$. For easier comparison with \cite{DyatlovWaveAsymptotics,VasyMicroKerrdS,WunschZworskiNormHypResolvent}, we consider the operator
\begin{equation}
\label{EqProofOp}
  P=-r^2\Box_{g_0} = -r^2\mu^{-1} D_t^2 + r^{-n+4}D_r r^{n-2}\mu D_r + \Delta_{\Sph^{n-2}}
\end{equation}
instead. We take the Fourier transform in $-t$ and rescale to a semiclassical operator on $X$ (this amounts to multiplying $\wh P$ by $h^2$, giving a second order semiclassical differential operator $P_h$, with $h=|\sigma|^{-1}$, and we then define $z=h\sigma$). Introducing coordinates on $T^*X$ by writing $1$-forms as $\xi\,dr+\eta\,d\omega$, and letting
\[
  \Delta_r = r^2\mu = r^2(1-\lambda r^2)-2M_\bullet r^{5-n},
\]
$P_h$ has semiclassical principal symbol
\[
  p = \Delta_r\xi^2 - \frac{r^4}{\Delta_r}z^2 + |\eta|^2,
\]
and correspondingly the Hamilton vector field is
\[
  \ham_p = 2\Delta_r\xi\pa_r - \Bigl(\pa_r\Delta_r\xi^2-\pa_r\Bigl(\frac{r^4}{\Delta_r}\Bigr)z^2\Bigr)\pa_\xi + \ham_{|\eta|^2}
\]
We work with real $z$, hence $z=\pm 1$. First, we locate the trapped set: If $\ham_p r=2\Delta_r\xi=0$, then $\xi=0$, in which case $\ham_p^2 r=2\Delta_r \ham_p\xi=2\Delta_r\pa_r(r^4/\Delta_r)z^2$. Recall the definition of the function $\wt\mu=\mu/r^2=\Delta_r/r^4$, then we can rewrite this as $\ham_p^2r=-2\Delta_r\wt\mu^{-2}(\pa_r\wt\mu)z^2$. We have already seen that $\pa_r\wt\mu$ has a single root $r_p\in(r_-,r_+)$, and $(r-r_p)\pa_r\wt\mu<0$ for $r\neq r_p$. Therefore, $\ham_p^2r=0$ implies (still assuming $\ham_p r=0$) $r=r_p$. We rephrase this to show that the only trapping occurs in the cotangent bundle over $r=r_p$: Let $F(r)=(r-r_p)^2$, then $\ham_p F=2(r-r_p)\ham_p r$ and $\ham_p^2F=2(\ham_p r)^2+2(r-r_p)\ham_p^2 r$. Thus, if $\ham_p F=0$, then either $r=r_p$, in which case $\ham_p^2 F=2(\ham_p r)^2>0$ unless $\ham_p r=0$, or $\ham_p r=0$, in which case $\ham_p^2 F=2(r-r_p)\ham_p^2 r>0$ unless $r=r_p$. So $\ham_p F=0,p=0$ implies either $\ham_p^2 F>0$ or $r=r_p,\ham_p r=0$, i.e.
\[
  (r,\omega;\xi,\eta) \in \Gamma_\semi := \Bigl\{(r_p,\omega;0,\eta) \colon \frac{r^4}{\Delta_r}z^2=|\eta|^2\Bigr\},
\]
so $\Gamma_\semi$ is the only trapping in $T^*X$, and $F$ is an escape function. We compute the linearization of the $\ham_p$-flow at $\Gamma_\semi$ in the normal coordinates $r-r_p$ and $\xi$, to wit
\begin{align*}
  \ham_p\begin{pmatrix}r-r_p \\ \xi\end{pmatrix} &= \begin{pmatrix} 0 & 2r_p^4\wt\mu|_{r=r_p} \\ 2(n-3)r_p^{-4}(\wt\mu|_{r=r_p})^{-2} z^2 & 0 \end{pmatrix}\begin{pmatrix} r-r_p \\ \xi \end{pmatrix} \\
	 & \quad + \cO(|r-r_p|^2+|\xi|^2),
\end{align*}
where we used $\pa_{rr}\wt\mu|_{r=r_p} = -2(n-3)r_p^{-4}$, which gives $\pa_r\wt\mu=-2(n-3)r_p^{-4}(r-r_p)+\cO(|r-r_p|^2)$. The eigenvalues of the linearization are therefore
\[
  \pm 2r_p\left(\frac{n-1}{1-\frac{n-1}{n-3}r_p^2\lambda}\right)^{1/2},
\]
which reduces to the expression given in \cite[p.\ 85]{VasyMicroKerrdS} in the case $n=4$, where $r_p=3M_\bullet=\frac{3}{2}r_s$ with $r_s=2M_\bullet$, and $\lambda=\Lambda/3$. In particular, the minimal expansion rate for the semiclassical rescaling of $\Box$ at the trapping $\Gamma_\semi$ is
\[
  \numin = 2r_p^{-1}\left(\frac{n-1}{1-\frac{n-1}{n-3}r_p^2\lambda}\right)^{1/2}>0.
\]
The expansion rate of the flow within the trapped set is $0$ by spherical symmetry; note that integral curves of $\ham_p$ on $\Gamma_\semi$ are simply unit speed geodesics of the round unit sphere $\Sph^{n-2}$. This shows the normal hyperbolicity (in fact, $r$-normal hyperbolicity for every $r$) of the trapping and finishes the proof of Theorem~\ref{ThmWaveExpansion}.

For later reference, we note that the spacetime trapped set, i.e.\ the set of points in phase space that never escape through either horizon along the Hamilton flow, is given by
\begin{equation}
\label{EqTrappedSet}
  \Gamma = \{ (t,r=r_p,\omega; \sigma,\xi=0,\eta)\colon \sigma^2=\Psi^2|\eta|^2 \},
\end{equation}
where $\Psi=\alpha r^{-1}$, $\Psi'(r_p)=0$.

\section{Pseudodifferential inner products}
\label{SecPsdoInner}

We now develop a general theory of pseudodifferential inner products, which we apply to the setting of Theorem~\ref{ThmWaveExpansion} in \S\ref{SecSubprincipalLaplace}.

We work on a complex rank $N$ vector bundle $\cE$ over the smooth compact $n$-dimensional manifold $X$ without boundary. We will define \emph{pseudodifferential inner products} on $\cE$, which are inner products depending on the position in phase space $T^*X$, rather than merely the position in the base $X$. As indicated in the introduction, we achieve this by replacing ordinary inner products by pseudodifferential operators whose symbols are inner products on the bundle $\pi^*\cE\to T^*X\setminus 0$, where $\pi\colon T^*X\setminus 0\to X$ is the projection.

In our application, we will use b-pseudodifferential inner products on tensor bundles over the spacetime manifold $\ol M$, but since the discussion in this section is purely symbolic, we work with standard \psdo{}s throughout; see also Remark~\ref{RmkOtherCalculi}.

\subsection{Notation}
\label{SubsecPsdoInnerNotation}

Let $\cV$ be a complex $N$-dimensional vector space. We denote by $\ol\cV$ the complex conjugate of $\cV$, i.e.\ $\ol\cV=\cV$ as sets, and the identity map $\iota\colon\cV\to\ol\cV$ is antilinear, so $\iota(\lambda v)=\ol\lambda\iota(v)$ for $v\in\cV$, $\lambda\in\C$, which defines the linear structure on $\ol\cV$. (We prefer to write $\iota(v)$ rather than $\ol v$ to prevent possible confusion with taking complex conjugates in complexifications of real vector spaces.) A Hermitian inner product $H$ on $\cV$ is thus a linear map $H\colon\cV\otimes\ol{\cV}\to\C$ such that $H(u,\iota(v))=\ol{H(v,\iota(u))}$ for $u,v\in\cV$, and $H(u,\iota(u))>0$ for all non-zero $u\in\cV$. This can be rephrased this in terms of the linear map $B\colon\cV\to\ol{\cV}^*$ defined by $B(u)=H(u,\cdot)$ and the natural dual pairing of $\ol{\cV}^*$ with $\ol\cV$, namely $\la Bu,\iota(v)\ra=\ol{\la Bv,\iota(u)\ra}$, and $\la Bu,\iota(u)\ra>0$ for all non-zero $u\in\cV$.

A linear map $A\colon\cV\to\ol{\cV}^*$ has an adjoint $A^*\colon\cV\to\ol{\cV}^*$, which is also linear, satisfying $\la Au,\iota(v)\ra=\ol{\la A^*v,\iota(u)\ra}$. The symmetry of a Hermitian inner product $B$ as above is then simply expressed by $B=B^*$. Similarly, a linear map $P\colon\cV\to\cV$ has an adjoint $P^*\colon\ol{\cV}^*\to\ol{\cV}^*$ defined by $\la\ol{\ell},\iota(Pv)\ra=\la P^*\ol{\ell},\iota(v)\ra$ for $\ol{\ell}\in\ol{\cV}^*$ and $v\in\cV$. These definitions of adjoints of maps $A\colon\cV\to\ol{\cV}^*$ and $P\colon\cV\to\cV$ are compatible in the sense that $(AP)^*=P^*A^*$. Furthermore, if $B\colon\cV\to\ol{\cV}^*$ is a Hermitian inner product and $Q\colon\cV\to\cV$ is invertible, then $B_1=Q^*BQ$ defines another Hermitian inner product, $\la B_1u,\iota(v)\ra = \la BQu,\iota(Qv)\ra$.

Now, given an inner product $B$ on $\cV$ and any map $P\colon\cV\to\cV$, the adjoint $P^{*B}$ of $P$ with respect to $B$ is the unique map $P^{*B}\colon\cV\to\cV$ such that $\la BPu,\iota(v)\ra=\la Bu,\iota(P^{*B}v)\ra$ for all $u,v\in\cV$. We find a formula for $P^{*B}$ by computing
\[
  \la BPu,\iota(v)\ra = \ol{\la B^*(B^*)^{-1}P^*B^*v,\iota(u)\ra} = \la Bu,\iota((BPB^{-1})^*v)\ra,
\]
i.e.\ $P^{*B}=(BPB^{-1})^*=B^{-1}P^*B$. The self-adjointness of $P$ with respect to $B$ is thus expressed by the equality $P=B^{-1}P^*B$.

If $\cE$ is a complex rank $N$ vector bundle, we can similarly define the complex conjugate bundle $\ol{\cE}$ as well as adjoints of vector bundle maps $\cE\to\cE$ and $\cE\to\ol{\cE}^*$. We can also define adjoints of pseudodifferential operators mapping between these bundles: For convenience, we remove the dependence of adjoints on a volume density on $X$ by tensoring all bundles with the half-density bundle $\Omega^\half$ over $X$, then we have a natural pairing
\[
  (\ol\cE^*\otimes\Omega^\half)_x \times (\ol\cE\otimes\Omega^\half)_x \ni (\ol\ell,\iota(v)) \mapsto \la\ol\ell,\iota(v)\ra \in \Omega^1_x, \quad x\in X,
\]
Thus, an operator $A\in\Psi^m(X,\cE\otimes\Omega^\half,\ol{\cE}^*\otimes\Omega^\half)$ has an adjoint $A^*\in\Psi^m(X,\cE\otimes\Omega^\half,\ol{\cE}^*\otimes\Omega^\half)$ defined by
\[
  \int_X \la A^*u,\iota(v)\ra = \int_X \ol{\la Av,\iota(u)\ra},
\]
with principal symbol $\sigma_m(A^*)=\sigma_m(A)^*\in S^m(T^*X\setminus 0,\pi^*\Hom(\cE,\ol{\cE}^*))$, and likewise $P\in\Psi^m(X,\cE\otimes\Omega^\half)$ has an adjoint $P^*\in\Psi^m(X,\ol{\cE}^*\otimes\Omega^\half)$ with $\sigma_m(P^*)=\sigma_m(P)^*$.

\subsection{Definition of pseudodifferential inner products; adjoints}
\label{SubsecPsdoInnerDef}

We work with classical, i.e.\ one-step polyhomogeneous, symbols and operators, and denote by $S^m_{\mathrm{hom}}(T^*X\setminus 0)$ symbols which are homogeneous of degree $m$ with respect to dilations in the fibers of $T^*X\setminus 0$.

\begin{definition}
\label{DefPsdoInnerProduct}
  A \emph{pseudodifferential inner product} (or $\Psi$-inner product) \emph{on the vector bundle $\cE\to X$} is a pseudodifferential operator $B\in\Psi^0(X;\cE\otimes\Omega^\half,\ol{\cE}^*\otimes\Omega^\half)$ satisfying $B=B^*$, and such that moreover the principal symbol $\sigma^0(B)=b\in S^0_{\mathrm{hom}}(T^*X\setminus 0;\pi^*\Hom(\cE,\ol{\cE}^*))$ of $B$ satisfies
  \begin{equation}
  \label{EqPsdoInnerPosDef}
    \la b(x,\xi)u,\iota(u)\ra > 0
  \end{equation}
  for all non-zero $u\in\cE_x$, where $x\in X$, $\xi\in T^*_x X\setminus 0$. If the context is clear, we will also call the sesquilinear pairing
  \[
    \CI(X,\cE\otimes\Omega^\half)\times\CI(X,\cE\otimes\Omega^\half) \ni (u,v) \mapsto \int_X \la B(x,D)u,\iota(v)\ra
  \]
  the pseudodifferential inner product associated with $B$.
\end{definition}

In particular, the principal symbol $b$ of $B$ is a Hermitian inner product on $\pi^*\cE$. Conversely, for any $b\in S^0_{\mathrm{hom}}(T^*X\setminus 0;\pi^*\Hom(\cE,\ol{\cE}^*))$ satisfying $b=b^*$ and \eqref{EqPsdoInnerPosDef}, there exists a $\Psi$-inner product $B$ with $\sigma^0(B)=b$; indeed, simply take $\wt B$ to be any quantization of $b$ and put $B=\frac{1}{2}(\wt B+\wt B^*)$.

\begin{rmk}
\label{RmkOtherCalculi}
  While we will develop the theory of $\Psi$-inner products only in the standard calculus on a closed manifold, everything works \emph{mutatis mutandis} in other settings as well. Thus, in the b-calculus of Melrose \cite{MelroseAPS}, $\Psib$-inner products on a manifold with boundary are defined similarly to $\Psi$-inner products, except that adjoints are defined on the space $\CIdot$ of functions vanishing to infinite order at the boundary, and the space of `trivial' (with respect to their symbolic order) operators is now $\Psib^{-\infty}$, likewise for the scattering calculus \cite{MelroseGeometricScattering}, replacing `b' by `sc.' In the semiclassical calculus on a closed manifold, adjoints are again defined on $\CI$, but the space of `trivial' operators is now $h^\infty\Psih^{-\infty}$, and suitable factors of $h$ need to be put in for computations involving subprincipal symbols.
\end{rmk}

We next discuss adjoints of \psdo{}s relative to $\Psi$-inner products.

\begin{definition}
\label{DefPsdoAdjoint}
  Let $B$ be a $\Psi$-inner product, and let $P\in\Psi^m(X,\cE\otimes\Omega^\half)$, then $P^{*B}\in\Psi^m(X,\cE\otimes\Omega^\half)$ is called an \emph{adjoint of $P$ with respect to $B$} if there exists an operator $R\in\Psi^{-\infty}(X,\cE\otimes\Omega^\half,\ol{\cE}^*\otimes\Omega^\half)$ such that
  \begin{equation}
  \label{EqPsdoAdjoint}
    \int \la BPu,\iota(v)\ra = \int \la Bu,\iota(P^{*B}v)\ra + \int \la Ru,\iota(v)\ra
  \end{equation}
  for all $u,v\in\CI(X,\cE\otimes\Omega^\half)$.
\end{definition}

\begin{rmk}
  This definition and the following lemma have straightforward generalizations to the case that $P$ maps sections of $\cE$ into sections of another vector bundle $\cF$, provided a ($\Psi$-)inner product on $\cF$ is given.
\end{rmk}

\begin{lemma}
\label{LemmaAdjoint}
  In the notation of Definition~\ref{DefPsdoAdjoint}, the adjoint of $P$ with respect to $B$ exists and is uniquely determined modulo $\Psi^{-\infty}(X,\cE\otimes\Omega^\half)$. In fact, $P=(BPB^-)^*$, where $B^-$ is a parametrix for $B$. Moreover, $(P^{*B})^{*B}=P$ modulo $\Psi^{-\infty}(X,\cE\otimes\Omega^\half)$. In particular, $\Im^B P=\frac{1}{2i}(P-P^{*B})$ is self-adjoint with respect to $B$ (i.e.\ its own adjoint modulo $\Psi^{-\infty}$).
\end{lemma}
\begin{proof}
  Let $B^-$ be a parametrix of $B$ and put $R_L=I-B^-B\in\Psi^{-\infty}(X,\cE\otimes\Omega^\half)$. Then
  \[
    \int \la BPu,\iota(v)\ra = \int \la BPB^- Bu,\iota(v)\ra + \la BPR_Lu,\iota(v)\ra,
  \]
  hence \eqref{EqPsdoAdjoint} holds with $P^{*B}=(BPB^-)^*$ and $R=BPR_L$. To show the uniqueness of $P^{*B}$ modulo smoothing operators, suppose that $\wt P$ is another adjoint of $P$ with respect to $B$, with error term $\wt R$ (i.e.\ \eqref{EqPsdoAdjoint} holds with $P^{*B}$ and $R$ replaced by $\wt P$ and $\wt R$). Then
  \begin{align*}
    \int\ol{\la B(P^{*B}-\wt P)v,\iota(u)\ra}&=\int \la Bu, \iota((P^{*B}-\wt P)v)\ra = \int \la (\wt R-R)u,\iota(v)\ra \\
	  &= \int\ol{\la(\wt R-R)^*v,\iota(u)\ra}
  \end{align*}
  for $u,v\in\CI(X,\cE\otimes\Omega^\half)$, so $B(P^{*B}-\wt P)=(\wt R-R)^*\in\Psi^{-\infty}(X,\cE\otimes\Omega^\half,\ol{\cE}^*\otimes\Omega^\half)$, and the ellipticity of $B$ implies $P^{*B}-\wt P\in\Psi^{-\infty}(X,\cE\otimes\Omega^\half)$, as claimed.

  Since $B$ is self-adjoint, we can assume that $B^-$ is self-adjoint by replacing it by $\frac{1}{2}(B^-+(B^-)^*)$ (which changes $B^-$ by an operator in $\Psi^{-\infty}$). Then the second claim follows from
  \[
    (P^{*B})^{*B} = (BP^{*B}B^-)^* = B^-BPB^-B = P
  \]
  modulo $\Psi^{-\infty}(X,\cE\otimes\Omega^\half)$.
\end{proof}

We note that self-adjointness on the operator level implies self-adjointness on the symbolic level:

\begin{lemma}
\label{LemmaSelfAdjointSymbol}
  Suppose $P\in\Psi^m(X,\cE\otimes\Omega^\half)$ is self-adjoint with respect to $B$. Then its principal symbol $p$ is self-adjoint with respect to $b=\sigma^0(B)$, i.e.
  \[
    \la b(x,\xi)p(x,\xi)u,\iota(v)\ra = \la b(x,\xi)u,\iota(p(x,\xi)v)\ra,\quad x\in X,\xi\in T_x X,u,v\in\cE_x.
  \]
\end{lemma}
\begin{proof}
  The hypothesis on $P$ means $(BPB^-)^*=P$ modulo $\Psi^{-\infty}$, thus on the level of principal symbols, $p=b^{-1}p^* b=p^{*b}$ (see \S\ref{SubsecPsdoInnerNotation} for the notation used), which proves the claim.
\end{proof}

We now specialize to the case that $P\in\Psi^m(X,\cE\otimes\Omega^\half)$ has a real, scalar principal symbol, which is the case of interest in our application, see \eqref{EqProofOp}. Fix a coordinate system of $X$ and a local trivialization of $\cE$, then the full symbol of $P$ is a sum of homogeneous symbols $p\sim p_m+p_{m-1}+\ldots$, with $p_j$ homogeneous of degree $j$ and valued in complex $N\times N$ matrices. We recall from \cite[\S{18}]{HormanderAnalysisPDE3} that the subprincipal symbol
\begin{equation}
\label{EqScalarSubprincipal}
  \sigma_\sub(P)=p_{m-1}(x,\xi)-\frac{1}{2i}\sum_j\pa_{x_j\xi_j}p_m(x,\xi) \in S^{m-1}_{\mathrm{hom}}(T^*X\setminus 0,\C^{N\times N})
\end{equation}
is well-defined under changes of coordinates; however, it does depend on the choice of local trivialization of $\cE$. As explained in \S\ref{SecProof}, we need to understand (the size of) the principal symbol of
\[
  \Im^B P:=\frac{1}{2i}(P-P^{*B})
\]
for such $P$ in a local trivialization of $\cE$. We first give a formula computing this symbol in a local trivialization; we will present an invariant formulation in Proposition~\ref{PropImagSymbolInv} below.

\begin{lemma}
\label{LemmaImagSymbol}
  Let $P\in\Psi^m(X,\cE\otimes\Omega^\half)$ be a principally real and scalar, and let $B=b(x,D)$ be a $\Psi$-inner product on $\cE$. Then $\Im^B P\in\Psi^{m-1}(X,\cE\otimes\Omega^\half)$ has the principal symbol
  \begin{equation}
  \label{EqImagSymbol}
    \sigma^{m-1}(\Im^B P)=\Im^b\sigma_\sub(P) + \frac{1}{2}b^{-1}\ham_p(b),
  \end{equation}
  where $\Im^b\sigma_\sub(P)=\frac{1}{2i}\bigl(\sigma_\sub(P)-\sigma_\sub(P)^{*b}\bigr)$. Here, we interpret $b$ and $\sigma_\sub(P)$ as $N\times N$ matrices of scalar-valued symbols using a local frame of $\cE$ and the corresponding dual frame of $\ol{\cE}^*$, and the action of $\ham_p$ is component-wise.
\end{lemma}

If the bundle was trivial, $\cE=X\times\C^N$, and $b$ the standard Hermitian inner product on $\cE$, i.e.\ $b$ is the $N\times N$ identity matrix with respect to standard bases, then \eqref{EqImagSymbol} simply states the well-known fact $\sigma^{m-1}(P-P^*)=\sigma_\sub(P)-\sigma_\sub(P)^*$. The presence of a non-trivial inner product $B$ causes an extra twist in the symbol of $P-P^{*B}$, which is the second term in \eqref{EqImagSymbol}.

\begin{proof}[Proof of Lemma~\ref{LemmaImagSymbol}]
  We compute in a local coordinate system over which $\cE$ and $\ol{\cE}$ are trivialized by a choice of $N$ linearly independent sections $e_1,\ldots,e_N$, and $\cE^*$ and $\ol{\cE}^*$ are trivialized by the dual sections $e_1^*,\ldots,e_N^*\in\cE^*$ satisfying $e_i^*(e_j)=\delta_{ij}$, extended linearly as linear functionals on $\cE$, resp.\ on $\ol{\cE}$, in the case of $\cE^*$, resp.\ $\ol{\cE}^*$. We trivialize $\Omega^\half$ using the section $|dx|^\half$. Let $b_{ij}(x,\xi)=\la b(x,\xi)e_j,\iota(e_i)\ra$, then $b(x,\xi)=(b_{ij}(x,\xi))_{i,j=1,\ldots,N}$, a linear map from the fibers of $\cE$ to the fibers of $\ol{\cE}^*$, is the symbol of $B$ in local coordinates: If $u=\sum_j u_je_j|dx|^\half$ and $v=\sum_j v_je_j|dx|^\half$, we have
  \[
    \la b(x,\xi)u,\iota(v)\ra = \sum_{ij} b_{ij}(x,\xi)u_j\cdot\ol{v_i}|dx|,
  \]
  thus
  \[
    \int \la Bu,\iota(v)\ra = \sum_{ij} \int (b_{ij}(x,D)u_j)\cdot\ol{v_j}\,dx.
  \]
  Note that $b(x,\xi)$ is a Hermitian matrix, i.e.\ $b_{ij}(x,\xi)=\ol{b_{ji}(x,\xi)}$, and in fact $B=b(x,D)$ is self-adjoint (with respect to the standard Hermitian inner product on $\C^N$). The adjoint of $P=p(x,D)$, which in local coordinates is simply an $N\times N$ matrix of scalar \psdo{}s, with respect to $B$ is the operator $\wt P=\wt p(x,D)$ such that
  \[
    \int b(x,D)p(x,D)u\cdot\ol{v}\,dx = \int b(x,D)u\cdot\ol{\wt p(x,D)v}\,dx + \int Ru\cdot\ol{v}\,dx,\quad R\in\Psi^{-\infty}.
  \]
  Let $B^-:=b^-(x,D)$ be a parametrix for $b(x,D)$, in particular $b^-(x,\xi)=b(x,\xi)^{-1}$ modulo $S^{-1}$; we may assume $B^-(x,D)^*=B^-(x,D)$. We then have
  \[
    \wt p(x,D) = b^-(x,D)p(x,D)^*b(x,D)
  \]
  by Lemma~\ref{LemmaAdjoint}. Write $p(x,\xi)=p_m(x,\xi)+p_{m-1}(x,\xi)+\ldots$, then the full symbol of $P-\wt P=B^-(BP-P^*B)$ (where $P^*$ is the adjoint of $P$ with respect to the standard Hermitian inner product on $\C^N$) is given, modulo $S^{m-2}$, by
  \begin{align*}
    b^{-1}&\Bigl(bp_m+\frac{1}{i}\sum_j\pa_{\xi_j}b\pa_{x_j}p_m + bp_{m-1} \\
	  & \quad - p_m^* b - \frac{1}{i}\sum_j(\pa_{x_j\xi_j}p_m^*)b - \frac{1}{i}\sum_j \pa_{\xi_j}p_m^*\pa_{x_j}b - p_{m-1}^*b\Bigr) \\
	  & = \Bigl(p_{m-1}-\frac{1}{2i}\sum_j\pa_{x_j\xi_j}p_m\Bigr) - b^{-1}\Bigl(p_{m-1}-\frac{1}{2i}\sum_j\pa_{x_j\xi_j}p_m\Bigr)^*b + ib^{-1}\ham_{p_m}(b),
  \end{align*}
  where we used that $p_m$ is scalar and real. The claim follows.
\end{proof}

\subsection{Invariant formalism for subprincipal symbols of operators acting on bundles}
\label{SubsecInvariantSubprincipal}

We continue to denote by $P\in\Psi^m(X,\cE\otimes\Omega^\half)$ a principally scalar \psdo{}\ acting on the vector bundle $\cE$, with principal symbol $p$; we remark that the discussion until Proposition~\ref{DefInvSubpr} works for principally non-scalar operators as well, with mostly notational changes. We will show how to modify the definition, given in equation~\eqref{EqScalarSubprincipal}, of the subprincipal symbol of $P$, expressed in terms of a local trivialization of $\cE$, in an invariant fashion, i.e.\ in a way that is both independent of the choice of local trivialization and of local coordinates on $X$. This provides a completely invariant formulation of Lemma~\ref{LemmaImagSymbol}. The advantage of being able to study principally scalar, subprincipally non-scalar operators in an invariant manner, apart from the naturality, is of course the freedom to choose particularly convenient local frames in concrete applications; for a warped product type spacetime geometry \eqref{EqSDSMetric} for instance, the cotangent bundle has a very natural partial frame.

Let $U\subset X$ be an open subset over which $\cE$ is trivial, and pick a frame $e(x)=\{e_1(x),\ldots,e_N(x)\}$ trivializing $\cE$ over $U$. Let us write $P^e$ for $P$ in the frame $e$, i.e.\ $P^e=(P^e_{jk})_{j,k=1,\ldots,N}$ is the $N\times N$ matrix of operators $P^e_{jk}\in\Psi^m(U,\Omega^\half)$ defined by $P(\sum_k u_k(x)e_k(x))=\sum_{jk}P^e_{jk}(u_k)e_j(x)$, $u_k\in\CI(U,\Omega^\half)$. Then $\sigma_\sub^e(P)$ as defined in \eqref{EqScalarSubprincipal}, with the superscript making the choice of frame explicit, is simply an $N\times N$ matrix of scalar symbols:
\[
  \sigma_\sub^e(P) = (\sigma_\sub(P^e_{jk}))_{j,k=1,\ldots,N}.
\]
We will consider the effect of a change of frame on the subprincipal symbol \eqref{EqScalarSubprincipal}. Thus, let $C\in\CI(U,\End(\cE))$ be a change of frame, i.e.\ $C(x)$ is invertible for all $x\in X$. Then $e_j(x)=C(x)e'_j(x)$ defines another frame $e'(x)=\{e'_1(x),\ldots,e'_N(x)\}$ of $\cE$ over $U$. One easily computes
\[
  \sigma_\sub^{e'}(C^{-1}PC) = (C^{e'})^{-1}\sigma_\sub^{e'}(P)C^{e'} - i(C^{e'})^{-1}\ham_p(C^{e'}),
\]
with $\ham_p$ interpreted as the diagonal $N\times N$ matrix $1_{N\times N} \ham_p$ of first order differential operators, and $C^{e'}$ is the matrix of $C$ in the frame $e'$. Now note that $(C^{-1}PC)^{e'}=P^e$ and $(C^{e'})^{-1}\ham_p(C^{e'})=(C^{e'})^{-1}\ham_p C^{e'} - \ham_p$; thus, we obtain
\begin{equation}
\label{EqInvSubprProof}
  \sigma_\sub^e(P) - i\ham_p = (C^{e'})^{-1}\bigl(\sigma_\sub^{e'}(P) - i\ham_p\bigr)C^{e'}
\end{equation}
Thus, viewing $\sigma_\sub^{e'}(P)-i\ham_p$ as the $N\times N$ matrix (in the frame $e'$) of a differential operator acting on $\CI(T^*X\setminus 0,\pi^*\cE)$, the right hand side of \eqref{EqInvSubprProof} is the matrix of the same differential operator, but expressed in the frame $e$. Notice that the principal symbol $p$ of $P$ as a scalar, i.e.\ diagonal, $N\times N$ matrix of symbols, is well-defined independently of the choice of frame. To summarize:

\begin{definition}
\label{DefInvSubpr}
  For $P\in\Psi^m(X,\cE\otimes\Omega^\half)$ with scalar principal symbol $p$, there is a well-defined \emph{subprincipal operator} $S_\sub(P)\in\Diff^1(T^*X\setminus 0,\pi^*\cE)$, homogeneous of degree $m-1$ with respect to dilations in the fibers of $T^*X\setminus 0$, defined as follows: If $\{e_1(x),\ldots,e_N(x)\}$ is a local frame of $\cE$, define the operators $P_{jk}\in\Psi^m(X,\Omega^\half)$ by $P(\sum_k u_k(x)e_k(x))=\sum_{jk} P_{jk}(u_k)e_j(x)$, $u_k\in\CI(X,\Omega^\half)$. Then
  \[
    S_\sub(P)\Bigl(\sum_k q_k(x,\xi)e_k(x)\Bigr) := \sum_{jk} (\sigma_\sub(P_{jk})q_k)e_j - i\sum_k(\ham_p q_k)e_k.
  \]
  In shorthand notation, $S_\sub(P) = \sigma_\sub(P) - i\ham_p$, understood in a local frame as a matrix of first order differential operators. We emphasize the dependence on the order of the operator by writing $S_{\sub,m}(P)$, so that for $P\in\Psi^m(X,\cE\otimes\Omega^\half)$, we have $S_{\sub,m+1}(P)=\sigma_m(P)$.
\end{definition}

In \S\ref{SubsecToyCase}, we present a very simple example of the formalism developed here. For our application, we will need to compute the subprincipal operator of the Laplace-Beltrami or Hodge-d'Alembert operator acting on sections of the tensor bundle, see \S\ref{SecSubprincipalLaplace}.

\begin{rmk}
  For $\Psib$-inner products, the subprincipal operator of an operator $P\in\Psib^m(X,\cE\otimes\Omega_\bl^\half)$ acting on $\cE$-valued b-half-densities is an element of $\Diffb^1(\Tb^*X\setminus 0,\pi_\bl^*\cE)$, where $\pi_\bl\colon\Tb^*X\setminus 0\to X$ is the projection. In the semiclassical setting, $P\in\Psih^m(X,\cE\otimes\Omega^\half)$, we have $S_\sub(P)\in\Diff^1(T^*X,\pi^*\cE)$; this is not a semiclassical operator, since the leading order part (the Hamilton derivative) and the zeroth order part (coming from the subprincipally non-scalar nature of $P$) are of the same size: For instance, \eqref{EqSubprImagPart} below for semiclassical operators $P,Q$ reads: $\sigma_\semi^{m+m'-1}(h^{-1}[P,Q])=[S_\sub(P),q]$.
\end{rmk}

\begin{rmk}
\label{RmkDencker}
  As already mentioned in \S\ref{SecIntro}, Dencker \cite{DenckerPolarization} proved that polarization sets propagate along so-called Hamilton orbits, which are line subbundles of the pullback of $\pi^*\cE$ to null-\-bi\-char\-ac\-ter\-istics, and which are spanned by sections of this bundle which are parallel with respect to a partial connection $D_P$. In the case of interest for us, when $P$ is principally scalar, his definition \cite[Equation~(4.6)]{DenckerPolarization}, choosing $\wt p=\id$, agrees with our definition of $S_\sub(P)$ up to a factor of $i$. Since $S_\sub(P)$ is only defined for principally scalar operators, whereas $D_P$ is defined for general operators of real principal type, but in general only up to rescaling (the motivation for introducing $D_P$ in \cite{DenckerPolarization} being quite different from the objective of the present paper), we use the notation $S_\sub(P)$ for clarity.
\end{rmk}

We can now express the symbols of commutators and imaginary parts in a completely invariant fashion:

\begin{prop}
\label{PropImagSymbolInv}
  Let $P\in\Psi^m(X,\cE\otimes\Omega^\half)$ be a \psdo{}\ with scalar principal symbol $p$.
  \begin{enumerate}[leftmargin=\enummargin]
  \item Suppose $Q\in\Psi^{m'}(X,\cE\otimes\Omega^\half)$ is an operator acting on $\cE$-valued half-densities, with principal symbol $q$. (We do not assume $Q$ is principally scalar.) Then
    \begin{equation}
    \label{EqSubprImagPart}
	  \sigma^{m+m'-1}([P,Q]) = [S_\sub(P),q].
	\end{equation}
	If $Q$ is elliptic with parametrix $Q^-$, then
	\begin{equation}
	\label{EqSubprChangeOfBasis}
	  S_\sub(QPQ^-)=q S_\sub(P)q^{-1}.
	\end{equation}
  \item Suppose in addition that $p$ is real. Let $B$ be a $\Psi$-inner product on $\cE$ with principal symbol $b$, then
    \begin{equation}
	\label{EqInvImagSymbol}
	  \sigma^{m-1}(\Im^B P) = \Im^b S_\sub(P),
	\end{equation}
	where $\Im^b S_\sub(P)=\frac{1}{2i}\bigl(S_\sub(P)-S_\sub(P)^{*b}\bigr)$; we take the adjoint of the differential operator $S_\sub(P)$ with respect to the inner product $b$ on $\pi^*\cE$ and the symplectic volume density on $T^*X$.
  \end{enumerate}
\end{prop}

In the context of Remark~\ref{RmkDencker}, one can check that \eqref{EqSubprImagPart} is equivalent to \cite[Equation~(4.7)]{DenckerPolarization}.

\begin{proof}[Proof of Proposition~\ref{PropImagSymbolInv}]
  We verify this in a local frame $e(x)=\{e_1(x),\ldots,e_N(x)\}$ of $\cE$. We compute
  \begin{align*}
	S_\sub&(P)\Bigl(\sum_{jk}q_{jk}(x,\xi)u_k(x,\xi)e_j(x)\Bigr) \\
	  &= \sum_{j\ell}\Bigl(\sum_k\sigma_\sub(P)_{jk}q_{k\ell} - i \ham_p(q_{j\ell})\Bigr)u_\ell e_j - iq_{j\ell}\ham_p(u_\ell)e_j - i q_{j\ell}u_\ell e_j \ham_p,
  \end{align*}
  while
  \begin{align*}
    q S_\sub&(P)\Bigl(\sum_\ell u_\ell(x,\xi)e_\ell(x)\Bigr) \\
      &= \sum_{j\ell}\Bigl(\sum_k q_{jk}\sigma_\sub(P)_{k\ell}\Bigr) u_\ell e_j - iq_{j\ell}\ham_p(u_\ell)e_j - iq_{j\ell}u_\ell e_j \ham_p,
  \end{align*}
  hence $S_\sub(P)q-qS_\sub(P) = [\sigma_\sub(P),q] - i\ham_p(q)$ as an endomorphism (a zeroth order differential operator acting on sections of $\cE$) of $\cE$ in the frame $e$, which equals $\sigma^{m+m'-1}([P,Q])$ according to the usual (full) symbolic calculus.

  Furthermore,
  \begin{align*}
    S_{\sub,m}&(QPQ^-)=S_{\sub,m}(P)+S_{\sub,m}(Q[P,Q^-]) \\
	  &=S_{\sub,m}(P) + q\sigma_{m+m'-1}([P,Q^-]) = S_{\sub,m}(P) + q[S_{\sub,m}(P),q^{-1}] \\
	  & = qS_{\sub,m}(P)q^{-1},
  \end{align*}
  noting that $Q[P,Q^-]$ is of order $m-1$.

  For the second part, we have $S_\sub(P)^{*b} = \sigma_\sub(P)^{*b} - (i\ham_p)^{*b} = b^{-1}\sigma_\sub(P)^* b + ib^{-1}(\ham_p)^* b$, where $(\ham_p)^*$ is the adjoint of $\ham_p$ as an operator acting on $\CIc(T^*X\setminus 0)$, and we equip $T^*X$ with the natural symplectic volume density $|dx\,d\xi|$. We have $(\ham_p)^* = -\ham_{\bar p} = -\ham_p$ since $p$ is real. Therefore,
  \begin{align*}
    S_\sub(P)-S_\sub(P)^{*b} &= \sigma_\sub(P)-\sigma_\sub(P)^{*b} - i\ham_p + ib^{-1}\ham_p b \\
	  & = \sigma_\sub(P)-\sigma_\sub(P)^{*b} + ib^{-1}\ham_p(b),
  \end{align*}
  which indeed gives \eqref{EqImagSymbol} upon division by $2i$.\qedhere
\end{proof}

In particular, \eqref{EqInvImagSymbol} provides a very elegant point of view for understanding the imaginary part of a principally scalar and real (pseudo)differential operator with respect to a $\Psi$-inner product $B$, as already indicated in the introduction: For instance, the principal symbol of the imaginary part $\Im^B P$ vanishes (or is small relative to $b=\sigma^0(B)$) in a subset of phase space if and only if the imaginary part of the first order differential operator $S_\sub(P)$ on $T^*X\setminus 0$ has vanishing (or small with respect to the fiber inner product $b$ of $\pi^*\cE$) coefficients in this subset.

\subsection{A simple example}
\label{SubsecToyCase}

On $\R^n_x=\R_{x_1}\times\R^{n-1}_{x'}$, we consider the operator $P=D_{x_1}+A\in\Psi^1(\R^n,\C^N)$, where $A=A(x,D)\in\Psi^0(\R^n,\C^N)$ is independent of $x_1$. Trivializing the half-density bundle over $\R^n$ via $|dx|^\half$, we can consider $P$ as an operator in $\Psi^1(\R^n,\C^N\otimes\Omega^\half)$. Its principal symbol is $\sigma_1(P)(x,\xi)=\xi_1$, where we use the standard coordinates on $T^*\R^n$, i.e.\ writing covectors as $\xi\,dx$, so the Hamilton vector field is $\ham_{\sigma_1(P)}=\pa_{x_1}$; moreover, in the trivialization of $\C^N$ by means of its standard basis, $\sigma_\sub(P)(x,\xi)=A(x,\xi)$. Thus, the subprincipal operator of $P$ is
\[
  S_\sub(P)(x,\xi)=A(x,\xi)-i\pa_{x_1} \in \Diff^1(T^*\R^n\setminus 0,\pi^*\C^N),
\]
with $A$ homogeneous of degree $0$ in the fiber variables. Suppose we are interested in bounding $\frac{1}{2i}(P-P^*)$ on $Z:=T^*_{\{x'=0\}}\R^n\setminus 0$ relative to a suitably chosen inner product. Let us assume that $A(0,\xi)$ is nilpotent for all $|\xi|=1$, and that in fact at $x=0$ and $|\xi|=1$, we can choose a smooth frame $e_1(\xi),\ldots,e_N(\xi)$ of the bundle $\pi^*\C^N\to T^*\R^n\setminus 0$ so that $A(0,\xi)$, written in the basis $e_1(\xi),\ldots,e_N(\xi)$, is a single Jordan block with zeros on the diagonal and ones directly above. Extend the $e_j$ by homogeneity (of degree $0$) in the fiber variables, and define them to be constant in the $x_1$-direction along $Z$, i.e.\ $e_j(x_1,0;\xi)=e_j(0,0;\xi)$, and extend them in an arbitrary manner to a neighborhood of $Z$. 

Now, on $Z$ we have $Ae_j=e_{j-1}$, writing $e_0:=0$. Introduce a new frame $e'_j:=\eps^j e_j$ with $\eps>0$ fixed, then $Ae'_j=\eps e'_{j-1}$. Define the inner product $b$ on $\pi^*\C^N$ by $\la b(x,\xi)(e'_i(x,\xi)),\iota(e'_j(x,\xi))\ra=\delta_{ij}$, that is, $\{e'_1,\ldots,e'_N\}$ is an orthonormal frame for $b$. Then on $Z$, we find that $\Im^b S_\sub(P)$ (which is of order $0$) in the frame $\{e'_1,\ldots,e'_N\}$ is given by the matrix which is zero apart from entries $\eps/2i$ directly above and $-\eps/2i$ directly below the diagonal. Thus, defining the $\Psi$-inner product $B=b(x,D)$, we have arranged that $\|\sigma_0(\Im^B P)(x,\xi)\|_b\leq\eps$ on $Z$. Since $\sigma_0(\Im^B P)$ is self-adjoint with respect to $b$, this is equivalent to the statement that its eigenvalues are bounded from above and below by $\eps$ and $-\eps$, respectively.

In Proposition~\ref{PropPsdoInnerAsConjugation} below we will show in general how to express the adjoint of an operator $P$ with respect to an $\Psi$-inner product as the ordinary adjoints of a conjugated version of $P$; in our example at hand, we can implement this very concretely as follows: If $v_j$ denotes the standard basis of $\C^N$ and $\la B_0(v_i),\iota(v_j)\ra=\delta_{ij}$ the standard inner product on $\C^N$ (the particular choice of an ordinary inner product being irrelevant, see the statement of Proposition~\ref{PropPsdoInnerAsConjugation}), define the map $q(x,\xi)\in S^0_{\mathrm{hom}}(T^*\R^n\setminus 0,\pi^*\C^N)$ by $q(x,\xi)e'_j(x,\xi)=v_j$. Let $Q=q(x,D)$ and denote by $Q^-$ a parametrix of $Q$, then we find that $QPQ^-\in\Psi^1(\R^n,\C^N)$ satisfies $\|\sigma_0(\Im^{B_0}QPQ^-)\|_{B_0}\leq\eps$.

If $A$ has several Jordan blocks not all of which are nilpotent, one can (under the assumption of the existence of a smooth family of Jordan bases) similarly construct a $\Psi$-inner product so that the imaginary part of $A$ relative to it is bounded by the maximal imaginary part of the eigenvalues of $A$ (plus $\eps$) from above, and by the minimal imaginary part (minus $\eps$) from below.

\subsection{Interpretation of pseudodifferential inner products in traditional terms}
\label{SubsecInterpretation}

Since the spectral gaps result \cite{DyatlovSpectralGaps} which we will invoke for our application is stated in terms of ordinary inner products, we now show how to interpret the imaginary part $\Im^B P$ of an operator $P$ with respect to a $\Psi$-inner product $B$ in terms of the imaginary part of a conjugated version of $P$ with respect to a standard inner product; we remark however that the proof of \cite[Theorem~1]{DyatlovSpectralGaps} would go through essentially unchanged if one used a $\Psi$-inner product directly.

\begin{prop}
\label{PropPsdoInnerAsConjugation}
  Let $B$ be a $\Psi$-inner product on $\cE$. Then for any positive definite Hermitian inner product $B_0\in\CI(X,\Hom(\cE\otimes\Omega^\half,\ol{\cE}^*\otimes\Omega^\half))$ on $\cE$, there exists an elliptic operator $Q\in\Psi^0(X,\End(\cE\otimes\Omega^\half))$ such that $B-Q^*B_0Q\in\Psi^{-\infty}(X,\Hom(\cE\otimes\Omega^\half,\ol{\cE}^*\otimes\Omega^\half))$.
  
  In particular, denoting by $Q^-\in\Psi^0(X,\End(\cE\otimes\Omega^\half))$ a parametrix of $Q$, we have for any $P\in\Psi^m(X,\cE\otimes\Omega^\half)$ with real and scalar principal symbol.
  \begin{equation}
  \label{EqPsdoInnerToNormal}
    Q(\Im^B P)Q^- = \Im^{B_0}(QPQ^-),
  \end{equation}
  and $\sigma^{m-1}(\Im^B P)$ and $\sigma^{m-1}(\Im^{B_0}(QPQ^-))$ (which are self-adjoint with respect to $\sigma^0(B)$ and $B_0$, respectively, hence diagonalizable) have the same eigenvalues.
\end{prop}

On a symbolic level, equation \eqref{EqPsdoInnerToNormal} is the same as equation~\eqref{EqSubprChangeOfBasis}.

\begin{proof}[Proof of Proposition~\ref{PropPsdoInnerAsConjugation}.]
  In order to shorten the notation, fix a global trivialization of $\Omega^\half$ over $X$ and use it to identify $\cE\otimes\Omega^\half$ with $\cE$, likewise for all other half-density bundles appearing in the statement. Denote the principal symbol of $B$ by $b\in S^0_{\mathrm{hom}}(T^*X\setminus 0,\pi^*\Hom(\cE,\ol{\cE}^*))$. We similarly put $b_0:=B_0$, which is an inner product on $\pi^*\cE$ that only depends on the base point.
  
  We start with on the symbolic level by constructing an elliptic symbol $q_1\in S^0_{\mathrm{hom}}(T^*X\setminus 0,\pi^*\End(\cE))$ such that $b=q_1^* b_0 q_1$; recall that $q_1^*\in S^0_{\mathrm{hom}}(T^*X\setminus 0,\pi^*\End(\ol{\cE}^*))$. For $t\in[0,1]$, define the Hermitian inner product $b_t:=(1-t)b_0+tb$. We will construct a differentiable family $q_t$ of symbols such that $b_t=q_t^* b_0 q_t$ for $t\in[0,1]$. Observe that for any such family, we have $\pa_t b_t=b-b_0 = (\pa_t q_t)^*b_0q_t + q_t^*b_0\pa_t q_t$, which suggests requiring $\pa_t q_t = \frac{1}{2}b_0^{-1}(q_t^*)^{-1}(b-b_0)$, which we can write as a linear expression in $q_t$ by noting that $(q_t^*)^{-1}=b_0 q_t b_t^{-1}$. Moreover, $q_0=\id$ is a valid choice for $q_t$ at $t=0$. Thus, we are led to define $q_t$, $t\in[0,1]$, as the solution of the ODE
  \[
    \pa_t q_t = \frac{1}{2}q_t b_t^{-1}(b-b_0),\quad q_0=\id.
  \]
  Reversing these arguments, for the solution $q_t$ we then have $q_t^*b_0 q_t=b_t$ for $t=0$, and both $q_t^*b_0 q_t$ and $b_t$ are solutions of the same ODE, namely
  \[
    \pa_t\wt b_t = \frac{1}{2}\bigl((b-b_0)b_t^{-1}\wt b_t + \wt b_t b_t^{-1}(b-b_0)\bigr), \quad \wt b_0=b_0,
  \]
  hence $q_t^*b_0 q_t=b_t$ for all $t\in[0,1]$.
  
  Let $Q_1\in\Psi^0(X,\End(\cE))$ be a quantization of $q_1$, then we conclude that $B-Q_1^* B_0Q_1\in\Psi^{-1}$. We iteratively remove this error to obtain a smoothing error: Suppose $Q_k\in\Psi^0(X,\End(\cE))$ is such that $B-Q_k^* B_0Q_k\in\Psi^{-k}$ for some $k\geq 1$. We will find $D_k\in\Psi^{-k}$, a quantization of $d_k\in S^{-k}_{\mathrm{hom}}(T^*X\setminus 0,\pi^*\cE)$, such that $Q_{k+1}:=Q_k+D_k$ satisfies $B-Q_{k+1}^* B_0Q_{k+1}\in\Psi^{-k-1}$. This is equivalent to the equality of symbols
  \[
     r_k:=\sigma^{-k}(B-Q_k^* B_0Q_k) = \sigma^{-k}(D_k^* B_0Q_k+Q_k^* B_0D_k) = d_k^* b_0q_1 + (b_0q_1)^* d_k,
  \]
  which in view of $r_k^*=r_k$ is satisfied for $d_k=\frac{1}{2}((b_0q_1)^*)^{-1}r_k$. We define $Q\in\Psi^0(X,\End(\cE))$ to be the asymptotic limit of the $Q_k$ as $k\to\infty$, i.e.\ $Q\sim Q_1+\sum_{k=1}^\infty D_k$, which thus satisfies $B-Q^* B_0Q\in\Psi^{-\infty}$. This proves the first part of the proposition.
  
  For the second part, denote parametrices of $B$ and $Q$ by $B^-$ and $Q^-$, respectively. Then, modulo operators in $\Psi^{-\infty}$, we have
  \[
    P^{*B} = (BPB^-)^* = (Q^* B_0Q PQ^- B_0^{-1} (Q^-)^*)^* = Q^-(QPQ^-)^{*B_0} Q,
  \]
  hence
  \[
    Q(P-P^{*B})Q^- = (QPQ^-) - (QPQ^-)^{*B_0}
  \]
  modulo $\Psi^{-\infty}$.
\end{proof}

\section{Subprincipal operators of tensor Laplacians}
\label{SecSubprincipalLaplace}

Let $(M,g)$ be a smooth manifold equipped with a metric tensor $g$ of arbitrary signature. Denote by $\cT_k M=\bigotimes^k T^*M$, $k\geq 1$, the bundle of (covariant) tensors of rank $k$ on $M$. The metric $g$ induces a metric (which we also call $g$) on $\cT_k M$. We study the symbolic properties of $\Delta_k=-\tr\nabla^2\in\Diff^2(M,\cT_k M)$, the Laplace-Beltrami operator on $M$ acting on the bundle $\cT_k M$. Denote by $G\in\CI(T^*M)$ the metric function, i.e.\ $G(x,\xi)=|\xi|_{G(x)}^2$, where $G$ is the dual metric of $g$.

\begin{prop}
\label{PropLaplaceSubpr}
  The subprincipal operator of $\Delta_k$ is
  \begin{equation}
  \label{EqLaplaceSubpr}
    S_\sub(\Delta_k)(x,\xi) = -i\conn{\pi^*\cT_k M}_{\ham_G} \in \Diff^1(T^*M\setminus 0,\pi^*\cT_k M),
  \end{equation}
  where $\conn{\pi^*\cT_k M}$ is the pullback connection, with $\pi\colon T^*M\setminus 0\to M$ being the projection.
\end{prop}
\begin{proof}
  Since both sides of \eqref{EqLaplaceSubpr} are invariantly defined, it suffices to prove the equality in an arbitrary local coordinate system. At a fixed point $x_0\in M$, introduce normal coordinates so that $\pa_k g_{ij}=0$ at $x_0$. Then we schematically have
  \begin{align*}
    (\Delta_k u)_{i_1\ldots i_k} &= -g^{jk}u_{i_1\ldots i_k,jk} = -g^{jk}(\pa_k u_{i_1\ldots i_k,j}+\Gamma\cdot\pa u) \\
	  &= -g^{jk}\pa_{jk}u_{i_1\ldots i_k} + \pa(\Gamma\cdot u) + \Gamma\cdot\pa u \\
	  &= -g^{jk}\pa_{jk}u_{i_1\ldots i_k} + \Gamma\cdot\pa u + \pa\Gamma\cdot u,
  \end{align*}
  with $\Gamma$ denoting Christoffel symbols. This suffices to see that the full symbol of $\Delta_k$ in the local coordinate system is given by
  \[
    \sigma(\Delta_k)(x,\xi) = g^{jk}(x)\xi_j\xi_k + (x^j-x_0^j)\ell_j(x,\xi) + e(x),
  \]
  where $\ell_j(x,\xi)$ is a linear map in $\xi$ with values in $\End((\cT_k M)_x)$, and $e(x)$ is an endomorphism of $(\cT_k M)_x$. Therefore, $\sigma_\sub(\Delta_k)(x_0,\xi)=0$, since $\pa_i g^{jk}(x_0)=0$. Thus,
  \begin{equation}
  \label{EqLaplaceSubprLHS}
    S_\sub(\Delta_k)(x_0,\xi) = -i \ham_{|\xi|_g^2} = -2i g^{jk}\xi_k\pa_{x^j}.
  \end{equation}
  We now compute the right hand side of \eqref{EqLaplaceSubpr}. First, writing $dx^I=dx^{i_1}\otimes\cdots\otimes dx^{i_k}$ for multi-indices $I=(i_1,\ldots,i_k)$, we note that sections of $\pi^*\cT_k M$ are of the form $u_I(x,\xi)\,dx^I$, while pullbacks (under $\pi$) of sections of $\cT_k M$ are of the form $u_I(x)\,dx^I$. By definition, the pullback connection $\conn{\pi^*\cT_k M}$ is given by
  \[
    \conn{\pi^*\cT_k M}_{\pa_{x^j}}(u_I(x)\,dx^I) = \conn{\cT_k M}_{\pa_{x^j}}(u_I(x)\,dx^I), \quad \conn{\pi^*\cT_k M}_{\pa_{\xi_k}}(u_I(x)\,dx^I)=0
  \]
  on pulled back sections and extended to sections of the pullback bundle using the Leibniz rule; thus,
  \begin{align*}
    \conn{\pi^*\cT_k M}_{\pa_{x^j}}(u_I(x,\xi)\,dx^I) &= \conn{\cT_k M}_{\pa_{x^j}}(u_I(\cdot,\xi)\,dx^I)(x), \\
	\conn{\pi^*\cT_k M}_{\pa_{\xi_k}}(u_I(x,\xi)\,dx^I) &= \pa_{\xi_k}u_I(x,\xi)\,dx^I.
  \end{align*}
  Thus, in normal coordinates at $x_0\in M$, we simply have $\conn{\pi^*\cT_k M}_{\pa_{x^j}}=\pa_{x^j}$ and $\conn{\pi^*\cT_k M}_{\pa_{\xi_k}}=\pa_{\xi_k}$, therefore
  \[
    \conn{\pi^*\cT_k M}_{\ham_{|\xi|_g^2}} = 2g^{jk}\xi_k\pa_{x^j}
  \]
  at $x_0$, which verifies \eqref{EqLaplaceSubpr} in view of \eqref{EqLaplaceSubprLHS}.
\end{proof}

To simplify the study of the pullback connection on $\pi^*\cT_k M$ for general $k$, we observe that there is a canonical bundle isomorphism $\pi^*\cT_k M\cong\bigotimes^k \pi^*T^*M$; hence the connection $\conn{\pi^*\cT_k M}$ is simply the product connection on $\bigotimes^k\pi^*T^*M$. Therefore, if we understand certain properties of $S_\sub(\Delta_1)$, we can easily deduce them for $S_\sub(\Delta_k)$ for any $k$. In our application, we will need to choose a \emph{positive definite} pseudodifferential inner product $B_k=b_k(x,D)$ on the bundle $\cT_k M$ with respect to which $\Delta_k$ is arbitrarily close to being symmetric in certain subsets of phase space. Concretely, this means that we want the operator $S_\sub(\Delta_k)$ to be (almost) symmetric with respect to the inner product $b_k$ on $\pi^*\cT_k M$. The following lemma shows that it suffices to accomplish this for $k=1$:

\begin{lemma}
\label{LemmaNilpotentInheritance}
  Let $U\subset T^*M\setminus 0$ be open, and let $f\in\CI(U)$ be real-valued. Fix a Hermitian inner product $b$ (antilinear in the second slot) on $\pi^*T^*M$, and define $R\in\End(\pi^*T^*M)$ by requiring that
  \[
    \int_U \la i\conn{\pi^*T^*M}_{\ham_f} u,v\ra_b\,d\sigma - \int_U \la u,i\conn{\pi^*T^*M}_{\ham_f} v\ra_b\,d\sigma = \int_U \la u,Rv\ra_b\,d\sigma
  \]
  for all $u,v\in\CIc(U,\pi^*T^*M)$, where $d\sigma$ is the natural symplectic volume density on $T^*M$. There exists a constant $C_k>0$, independent of $U,f$ and $b$, such that the following holds: If $\sup_U \|R\|_b\leq\eps$ (using $b$ to measure the operator norm of $R$ acting on each fiber) for some $\eps>0$, then the inner product $b_k=\bigotimes^k b$ induced by $b$ on $\bigotimes^k\pi^* T^*M\cong\pi^*\cT_k M$ satisfies
  \[
    \int_U \la i\conn{\pi^*\cT_k M}_{\ham_f} u,v\ra_{b_k}\,d\sigma - \int_U \la u,i\conn{\pi^*\cT_k M}_{\ham_f} v\ra_{b_k}\,d\sigma = \int_U \la u,R_k v\ra_{b_k}\,d\sigma,
  \]
  $u,v\in\CIc(U,\pi^*\cT_k M)$, for $R_k\in\End(\pi^*\cT_k M)$ satisfying $\sup_U \|R_k\|_{b_k}\leq k\eps$.
\end{lemma}
\begin{proof}
  We show this for $k=2$, the proof for general $k$ being entirely analogous. Denote $S=i\conn{\pi^*T^*M}_{\ham_f}$, then $S_2=i\conn{\pi^*\cT_2 M}_{\ham_f}$ acts by $S_2(u_1\otimes u_2)=Su_1\otimes u_2+u_1\otimes Su_2$. Hence using $S(au)=aSu + i\ham_f(a)u$ for sections $u$ of $\pi^*T^*M$ and functions $a$ on $U$, we calculate
  \begin{align*}
    \int_U &\la S_2(u_1\otimes u_2),v_1\otimes v_2\ra_{b_2}\,d\sigma = \int_U \la Su_1,v_1\ra_b \la u_2,v_2\ra_b + \la u_1,v_1\ra_b\la Su_2,v_2\ra_b\,d\sigma \\
	  &= \int_U \Big\la u_1, S\bigl(v_1\la u_2,v_2\ra_b\bigr)\Big\ra_b + \int_U \Big\la u_2,S\bigl(v_2\la u_1,v_1\ra_b\bigr)\Big\ra_b\,d\sigma \\
	  &\qquad + \int_U \la u_1\otimes u_2, (R\otimes\id+\id\otimes R)(v_1\otimes v_2)\ra_{b_2}\,d\sigma \\
	  &= \int_U \la u_1\otimes u_2,S_2(v_1\otimes v_2)\ra_{b_2}\,d\sigma - i\int_U \ham_f(\la u_1,v_1\ra_b\la u_2,v_2\ra_b)\,d\sigma \\
	  &\qquad + \int_U \la u_1\otimes u_2,R_2(v_1\otimes v_2)\ra_{b_2}\,d\sigma \\
	  &= \int_U \la u_1\otimes u_2,S_2(v_1\otimes v_2)\ra_{b_2}\,d\sigma + \int_U \la u_1\otimes u_2,R_2(v_1\otimes v_2)\ra_{b_2}\,d\sigma
  \end{align*}
  with $R_2=R\otimes\id+\id\otimes R$, where we used that $\int_U\ham_f u\,d\sigma=-\int_U u\ham_f 1\,d\sigma=0$ for $u\in\CIc(U)$. From the explicit form of $R_2$, we see that $\|R_2\|_{b_2}\leq 2\eps$ indeed.
\end{proof}

\subsection{Warped product spacetimes}
\label{SubsecWarpedSubpr}

Let $X$ be an $(n-1)$-dimensional manifold equipped with a smooth Riemannian metric $h=h(x,dx)$, and let $\alpha\in\CI(X)$ be a positive function. We consider the manifold $M=\R_t\times X$, equipped with the Lorentzian metric
\begin{equation}
\label{EqMetricGeneral}
  g=\alpha^2\,dt^2-h.
\end{equation}
On such a spacetime, we have a natural splitting of 1-forms into their tangential and normal part relative to $\alpha\,dt$, i.e.
\begin{equation}
\label{EqFormDecomp}
  u = u_T + u_N\alpha\,dt.
\end{equation}
In this section, we will compute the form of $\conn{\pi^*T^*M}_{\ham_G}$ as a $2\times 2$ matrix of differential operators with respect to this decomposition. For brevity, we will use the notation $\pbconn{M}:=\conn{\pi^*T^*M}$, similarly $\pbconn{X}:=\conn{\pi^*T^*X}$, and we will moreover use the abstract index notation, fixing $x^0=t$, and $x'=(x^1,\ldots,x^{n-1})$ are coordinates on $X$ (independent of $t$). We let Greek indices $\mu,\nu,\lambda,\ldots$ run from $0$ to $n-1$, Latin indices $i,j,k,\ldots$ from $1$ to $n-1$. Moreover, the canonical dual variables $\xi_0=:\sigma$ and $\xi'=(\xi_1,\ldots,\xi_{n-1})$ on the fibers of $T^*M$ are indexed by decorated Greek indices $\wt\mu$ (running from $0$ to $n-1$) and Latin indices $\wt i,\wt j,\ldots$ (running from $1$ to $n-1$). If an index appears both with and without tilde in one expression, it is summed accordingly, for instance $a_j b_{\wt j}=\sum_{j=1}^n a_j b_{\wt j}$. Thus, for a section $u$ of $\pi^*T^*M$, we have
\[
  \pbconn{M}_\mu u_\nu = \conn{M}_\mu u_\nu, \quad \pbconn{M}_{\wt\mu} u_\nu = \pa_{\wt\mu}u_\nu,
\]
where we interpret $\conn{M}_\mu$ as acting on $u$ for fixed values of the fiber variables, i.e.\ viewing $u$ as a family of sections of $T^*M$ depending on the fiber variables. As before, we denote by $G$ the metric function on $T^*M$, and we let $H$ denote the metric function on $T^*X$, interpreted as a $(t,\sigma)$-independent function on $T^*M$. Lastly, we denote the Christoffel symbols of $(M,g)$ by $\chr{M}_{\mu\nu}^\kappa$, and those of $(X,h)$ by $\chr{X}_{ij}^k$.

We point out that once we discuss Schwarzschild-de Sitter space in the next section, in the region where $t_*=t$ (which we can in particular arrange near the trapped set), $\sigma$ in the present notation is equal to $-\sigma$ in the notation of \S\ref{SecProof}.

\begin{lemma}
\label{LemmaChristoffel}
  The Christoffel symbols of $M$ are given by:
  \begin{equation}
  \label{EqChristoffel}
  \begin{gathered}
    \chr{M}_{00}^0=0,\quad	\chr{M}_{i0}^0=\alpha^{-1}\alpha_i, \quad \chr{M}_{ij}^0=0, \\
	\chr{M}_{00}^k = \alpha h^{k\ell}\alpha_\ell, \quad \chr{M}_{i0}^k = 0, \quad \chr{M}_{ij}^k=\chr{X}_{ij}^k.
  \end{gathered}
  \end{equation}
\end{lemma}
\begin{proof}
  We have $g_{00}=\alpha^2$, $g_{0i}=g_{i0}=0$ and $g_{ij}=-h_{ij}$, and $g$ is $t$-independent, thus $\pa_0 g_{\mu\nu}=0$. Using $\chr{M}_{\kappa\mu\nu} = \frac{1}{2}(\pa_\mu g_{\kappa\nu} + \pa_\nu g_{\kappa\mu} - \pa_\kappa g_{\mu\nu})$, we then compute
  \begin{gather*}
    \chr{M}_{000}=0, \quad \chr{M}_{0i0}=\alpha\alpha_i, \quad \chr{M}_{0ij}=0, \\
	\chr{M}_{k00}=-\alpha\alpha_k, \quad \chr{M}_{ki0}=0, \quad \chr{M}_{kij}=-\chr{X}_{kij},
  \end{gather*}
  which immediately gives \eqref{EqChristoffel}.
\end{proof}

\begin{prop}
\label{PropWarpedSubpr}
  For the metric $g$ as in \eqref{EqMetricGeneral}, the subprincipal operator of $\Box_1$ (the tensor wave operator acting on $1$-forms on $M$) in the decomposition \eqref{EqFormDecomp} of 1-forms is given by
  \begin{align*}
    i S_\sub(&\Box_1)(t,x',\sigma,\xi') \\
	 &= \begin{pmatrix}
	     2\alpha^{-2}\sigma\pa_t + \sigma^2\pbconn{X}_{\ham_{\alpha^{-2}}}-\pbconn{X}_{\ham_H} & -2\alpha^{-2}\sigma\,d\alpha \\
		 -2\alpha^{-2}\sigma i_{\conn{X}\alpha} & 2\alpha^{-2}\sigma\pa_t + \sigma^2\ham_{\alpha^{-2}}-\ham_H
	   \end{pmatrix}.
  \end{align*}
\end{prop}
\begin{proof}
  We start by computing the form of $\pbconn{M}_\mu u_\nu$ and $\pbconn{M}_{\wt\mu} u_\nu$ for tangential and normal 1-forms. For tangential forms $u=u_\mu\,dx^\mu$ with $u_0=0$, we have
  \begin{gather*}
    \pbconn{M}_0 u_0 = -\chr{M}_{00}^\lambda u_\lambda = -\alpha\la d\alpha,u\ra_H, \quad \pbconn{M}_0 u_i = \pa_0 u_i, \\
	\pbconn{M}_j u_0 = 0, \quad \pbconn{M}_j u_i = \conn{X}_j u_i, \quad \pbconn{M}_{\wt\mu} u_0 = 0, \quad \pbconn{M}_{\wt\mu} u_i = \pa_{\wt\mu} u_i,
  \end{gather*}
  while for normal forms $u=u_\mu\,dx^\mu$ with $u_i=0$ and $u_0=\alpha v$, we compute
  \begin{gather*}
    \pbconn{M}_0 u_0 = \alpha\pa_t v, \quad \pbconn{M}_0 u_i = -\alpha_i v, \\
	\pbconn{M}_j u_0 = \pa_j(\alpha v) - \alpha_j v=\alpha\pa_j v, \quad \pbconn{M}_j u_i = 0, \quad \pbconn{M}_{\wt\mu}u_0=\alpha\pa_{\wt\mu}v, \quad \pbconn{M}_{\wt\mu} u_i=0.
  \end{gather*}
  Since $G=\alpha^{-2}\sigma^2-H$, we find $\ham_G=2\alpha^{-2}\sigma\pa_t + \sigma^2\ham_{\alpha^{-2}} - \ham_H$. Using $\la d\alpha,\cdot\ra_H=i_{\conn{X}\alpha}$, we obtain
  \[
    \pbconn{M}_{\pa_t}=
	  \begin{pmatrix}
	    \pa_t & -d\alpha \\
		-i_{\conn{X}\alpha} & \pa_t
	  \end{pmatrix}.
  \]
  Moreover, for any $f\in\CI(T^*X)$ (we will take $f=\alpha^{-2}$ and $f=H$), viewed as a $(t,\sigma)$-independent function on $T^*M$, we have $\ham_f=f_{\wt j}\pa_j-f_j\pa_{\wt j}$. Hence on tangential forms,
  \[
    \pbconn{M}_{\ham_f} u_0 = 0, \quad \pbconn{M}_{\ham_f} u_i = f_{\wt j}\conn{X}_j u_i - f_j\pa_{\wt j}u_i = \pbconn{X}_{\ham_f} u_i,
  \]
  while on normal forms as above,
  \[
    \pbconn{M}_{\ham_f} u_0 = \alpha f_{\wt j}\pa_j v - \alpha f_j\pa_{\wt j}v = \alpha\ham_f v, \quad \pbconn{M}_{\ham_f} u_i = 0.
  \]
  Thus,
  \[
    \pbconn{M}_{\ham_f}=
	  \begin{pmatrix}
	    \pbconn{X}_{\ham_f} & 0 \\
		0 & \ham_f
	  \end{pmatrix}.
  \]
  The claim follows.
\end{proof}

\subsection{Schwarzschild-de Sitter space}
\label{SubsecSDS}

We stay in the setting of the previous section, and now the spatial metric $h$ has a decomposition
\[
  h = \alpha^{-2}\,dr^2 + r^2\,d\omega^2,
\]
where $d\omega^2$ is the round metric on the unit sphere $Y=\Sph^{n-2}$, with dual metric denoted $\Omega$; see \eqref{EqSDSMetric}. Thus, writing $\xi$, resp.\ $\eta$, for the dual variables of $r$, resp.\ $\omega\in\Sph^{n-2}$, we have $H=\alpha^2\xi^2+r^{-2}|\eta|_\Omega^2$. Write 1-forms on $X$ as
\begin{equation}
\label{EqFormDecompSDS}
  u = u_T + u_N\alpha^{-1}\,dr.
\end{equation}
Abbreviate the derivative of a function $f$ with respect to $r$ by $f'$. Since $d\alpha=\alpha'\,dr$ and $\conn{X}\alpha=\alpha^2\alpha'\pa_r$, we have, in the decomposition \eqref{EqFormDecompSDS},
\[
  d\alpha = \begin{pmatrix} 0 \\ \alpha\alpha' \end{pmatrix}, \quad i_{\conn{X}\alpha} = \begin{pmatrix} 0 & \alpha\alpha' \end{pmatrix}.
\]

We will need the Christoffel symbols of $h$. We continue using the notation to the previous section, except now $x^1=r$ and $\xi_1=\xi$, while $x^2,\ldots,x^n$ are $r$-independent coordinates on $\Sph^{n-2}$, and moreover the lower bound for Greek indices is $1$, and $2$ for Latin indices.
\begin{lemma}
\label{LemmaChristoffelSDS}
  The Christoffel symbols of $X$ are given by:
  \begin{equation}
  \label{EqChristoffelSDS}
  \begin{gathered}
    \chr{X}_{11}^1 = -\alpha^{-1}\alpha', \quad \chr{X}_{i1}^1 = 0, \quad \chr{X}_{ij}^1 = -r\alpha^2(d\omega^2)_{ij}, \\
	\chr{X}_{11}^k = 0, \quad \chr{X}_{i1}^k = r^{-1}\delta_i^k, \quad \chr{X}_{ij}^k = \chr{Y}_{ij}^k.
  \end{gathered}
  \end{equation}
\end{lemma}
\begin{proof}
  We have $h_{11}=\alpha^{-2}$, $h_{1i}=h_{i1}=0$ and $h_{ij}=r^2(d\omega^2)_{ij}$, and $(d\omega^2)_{ij}$ is $r$-independent. We then compute
  \begin{gather*}
    \chr{X}_{111} = -\alpha^{-3}\alpha', \quad \chr{X}_{1i1} = 0, \quad \chr{X}_{1ij} = -r(d\omega^2)_{ij}, \\
    \chr{X}_{k11} = 0, \quad \chr{X}_{ki1} = r(d\omega^2)_{ki}, \quad \chr{X}_{kij} = r^2\chr{Y}_{kij},
  \end{gather*}
  which immediately gives \eqref{EqChristoffelSDS}.
\end{proof}

We are only interested in the subprincipal operator of $\Box_1$ at the trapped set, which we recall from \eqref{EqTrappedSet} to be the set
\begin{equation}
\label{EqTrappedSetRecall}
  \Gamma = \{ r=r_p, \xi=0, \sigma^2=\Psi^2|\eta|^2 \}, \quad \tn{where }\Psi=\alpha r^{-1}, \Psi'(r_p)=0.
\end{equation}
Thus, at $\Gamma$, we have
\[
  \ham_H = 2\alpha^2\xi\pa_r-2\alpha\alpha'\xi^2\pa_\xi + 2r^{-3}|\eta|^2\pa_\xi + r^{-2}\ham_{|\eta|^2} = 2r^{-3}|\eta|^2\pa_\xi + r^{-2}\ham_{|\eta|^2},
\]
while $\sigma^2\ham_{\alpha^{-2}}=2\sigma^2\alpha^{-3}\alpha'\pa_\xi$. Now $\alpha^{-1}\alpha'=(r\Psi)^{-1}(r\Psi)'=r^{-1}$ at $r=r_p$, therefore $\sigma^2\alpha^{-3}\alpha'=r^{-3}|\eta|^2$, and we thus obtain
\begin{equation}
\label{EqHamAtTrapping}
  \sigma^2\ham_{\alpha^{-2}}-\ham_H = -r^{-2}\ham_{|\eta|^2}\tn{ at }\Gamma.
\end{equation}
Notice that $|\eta|^2\in\CI(T^*Y)$ is independent of $(r,\xi)$.
\begin{lemma}
\label{LemmaSphericalSliceSubpr}
  For a function $f\in\CI(T^*Y)$, viewed as an $(r,\xi)$-independent function on $X$, we have
  \[
    \pbconn{X}_{\ham_f}=
	  \begin{pmatrix}
	    \pbconn{Y}_{\ham_f} & \alpha r(i_{\ham_f}d\omega^2) \\
		-\alpha r^{-1} i_{\ham_f} & \ham_f
	  \end{pmatrix}.
  \]
  in the decomposition \eqref{EqFormDecompSDS} of 1-forms on $X$.
\end{lemma}
\begin{proof}
  On tangential forms $u$, i.e.\ $u_1=0$, we have
  \[
    \pbconn{X}_j u_1=-r^{-1}u_j, \quad \pbconn{X}_j u_i = \conn{Y}_j u_i, \quad \pbconn{X}_{\wt j}u_1=0, \quad \pbconn{X}_{\wt j}u_i = \pa_{\wt j}u_i,
  \]
  thus in view of $\ham_f=f_{\wt j}\pa_j - f_j\pa_{\wt j}$, we get, using that $\pi^*T^*X$ can be canonically identified with the horizontal subbundle of $T^*(T^*X)$:
  \[
    \pbconn{X}_{\ham_f}u_1 = -r^{-1}f_{\wt j}u_j = -r^{-1}u(\ham_f) = -r^{-1}i_{\ham_f} u, \quad \pbconn{X}_{\ham_f}u_i = \pbconn{Y}_{\ham_f}u_i.
  \]
  On normal forms $u$, i.e.\ $u_1=\alpha^{-1}v$, $u_i=0$, we compute
  \[
    \pbconn{X}_j u_1=\alpha^{-1}\pa_j v, \quad \pbconn{X}_j u_i=r\alpha(d\omega^2)_{ij}v, \quad \pbconn{X}_{\wt j}u_1=\alpha^{-1}\pa_{\wt j}v, \quad \pbconn{X}_{\wt j}u_i=0,
  \]
  hence
  \begin{gather*}
    \pbconn{X}_{\ham_f}u_1 = \alpha^{-1}f_{\wt j}\pa_j v - \alpha^{-1}f_j\pa_{\wt j}v = \alpha^{-1}\ham_f v, \\
	\pbconn{X}_{\ham_f}u_i=f_{\wt j} r\alpha(d\omega^2)_{ij}v = \alpha r(i_{\ham_f}d\omega^2)v.
  \end{gather*}
  The claim follows immediately.
\end{proof}

Combining Proposition~\ref{PropWarpedSubpr} and Lemma~\ref{LemmaSphericalSliceSubpr}, we can thus compute the subprincipal operator of $\Box_1$ acting on 1-forms (sections of the pullback of $T^*M$ to $T^*M\setminus 0$) decomposed as
\begin{equation}
\label{EqFormDecompFull}
  u = u_{TT} + u_{TN}\alpha^{-1}\,dr + u_N \alpha\,dt.
\end{equation}
In view of \eqref{EqHamAtTrapping}, we merely need to apply Lemma~\ref{LemmaSphericalSliceSubpr} to $f=|\eta|^2$, in which case $\ham_f=2\Omega^{jk}\eta_j\pa_k - \pa_\ell\Omega^{jk}\eta_j\eta_k\pa_{\wt\ell}$, so $i_{\ham_f}=2i_\eta$ on 1-forms (identifying the 1-form $\eta$ with a tangent vector using the metric $d\omega^2$), while $i_{\ham_f}d\omega^2=2\eta$. Thus, we obtain:
\begin{prop}
\label{PropSDSSubpr}
  In the decomposition~\eqref{EqFormDecompFull}, the subprincipal operator of $\Box_1$ on Schwarzschild-de Sitter space at the trapped set $\Gamma$ is given by
  \begin{equation}
  \label{EqSDSSubpr}
  \begin{split}
    i S_\sub(&\Box_1) \\
	  &=\begin{pmatrix}
	    2\alpha^{-2}\sigma\pa_t - r^{-2}\pbconn{Y}_{\ham_{|\eta|^2}} & -2\alpha r^{-1}\eta & 0 \\
		2\alpha r^{-3}i_\eta & 2\alpha^{-2}\sigma\pa_t - r^{-2}\ham_{|\eta|^2} & -2r^{-1}\sigma \\
		0 & -2r^{-1}\sigma & 2\alpha^{-2}\sigma\pa_t - r^{-2}\ham_{|\eta|^2}
	  \end{pmatrix}.
  \end{split}
  \end{equation}
\end{prop}
Since $\Box_1$ is symmetric with respect to the natural inner product $G$ on the 1-form bundle, which in the decomposition \eqref{EqFormDecompFull} is an orthogonal direct sum of inner products, $G=(-r^{-2}\Omega) \oplus (-1) \oplus 1$, the operator $S_\sub(\Box_1)$ is a symmetric operator acting on sections of $\pi^*T^*M$ over $T^*M\setminus 0$ if we equip $\pi^*T^*M$ with the fiber inner product $G$ and use the symplectic volume density on $T^*M\setminus 0$.

The matrix $-2r^{-2} s$, with
\[
  s = \begin{pmatrix}
        0 & \Psi r^2\eta & 0 \\
		-\Psi i_\eta & 0 & r\sigma \\
		0 & r\sigma & 0
      \end{pmatrix},
\]
of $0$-th order terms of $S_\sub(\Box_1)$ is nilpotent, which suggests in analogy to the discussion in \S\ref{SubsecToyCase} that the imaginary part of $S_\sub(\Box_1)$ with respect to a \emph{Riemannian} fiber inner product can be made arbitrarily small. Indeed, for any fixed $\eps>0$, define the `change of basis matrix'
\[
  q=\begin{pmatrix}
      \id & 0 & 0 \\
	  0 & \eps^{-1}\Psi r^2 & 0 \\
	  -\eps^{-2}|\eta|^{-1}\Psi^2 r^2 i_\eta & 0 & \eps^{-2}|\eta|^{-1}\Psi r^3\sigma
    \end{pmatrix},
\]
then
\[
  qsq^{-1}
  =\begin{pmatrix}
     0 & \eps\eta & 0 \\
	 0 & 0 & \eps|\eta| \\
	 0 & 0 & 0
   \end{pmatrix}.
\]
In order to compute $q S_\sub(\Box_1) q^{-1}$, we note that the diagonal matrix of $t$-derivatives in \eqref{EqSDSSubpr} commutes with $q$, and it remains to study the derivatives along $\ham_{|\eta|^2}$; more specifically, $q$ has a block structure, with the columns and rows $1,3$ being the first block and the $(2,2)$ entry the second, and the $(2,2)$ block is an $\eta$-independent multiple of the identity, hence commutes with the relevant $(2,2)$ entry $i r^{-2}\ham_{|\eta|^2}$ of $S_\sub(\Box_1)$. For the $1,3$ block, we compute
\begin{equation}
\label{EqBlock13}
\begin{split}
  \Biggl[\begin{pmatrix}\pbconn{Y}_{\ham_{|\eta|^2}} & 0 \\ 0 & \ham_{|\eta|^2} \end{pmatrix},& \begin{pmatrix} \id & 0 \\ -\eps^{-2}|\eta|^{-1}\Psi^2 r^2i_\eta & \eps^{-2}|\eta|^{-1}\Psi r^3\sigma \end{pmatrix} \Biggr] \\
  &\quad = \eps^{-2}\Psi^2 r^2|\eta|^{-1} \begin{pmatrix} 0 & 0 \\ i_\eta\pbconn{Y}_{\ham_{|\eta|^2}}-\ham_{|\eta|^2}i_\eta & 0 \end{pmatrix}.
\end{split}
\end{equation}
Now $\pbconn{Y}_{\ham_{|\eta|^2}}$ and $\ham_{|\eta|^2}$ are the restrictions of the pullback connection $\conn{\pi^*\Lambda\Sph^{n-2}}_{\ham_{|\eta|^2}}$ of the full form bundle to 1-forms and functions, respectively, and the latter commutes with $i_\eta$, since by Proposition~\ref{PropImagSymbolInv},
\[
  0 = S_\sub([\Box,\delta]) = -i[S_\sub(\Box),i_\eta] = -\bigl[\conn{\pi^*\Lambda\Sph^{n-2}}_{\ham_{|\eta|^2}},i_\eta\bigr],
\]
where $\Box$ denotes the Hodge d'Alembertian on the form bundle and $\delta$ is the codifferential. Thus, \eqref{EqBlock13} in fact vanishes, and therefore
\begin{align*}
  q &S_\sub(\Box_1)q^{-1} \\
  &=-i\begin{pmatrix}
    2\alpha^{-2}\sigma\pa_t - r^{-2}\pbconn{Y}_{\ham_{|\eta|^2}} & -2r^2\eps\eta & 0 \\
	0 & 2\alpha^{-2}\sigma\pa_t - r^{-2}\ham_{|\eta|^2} & -2r^2\eps|\eta| \\
	0 & 0 & 2\alpha^{-2}\sigma\pa_t - r^{-2}\ham_{|\eta|^2}
  \end{pmatrix}.
\end{align*}
Equip the 1-form bundle over $M$ in the decomposition~\eqref{EqFormDecompFull} with the Hermitian inner product
\begin{equation}
\label{EqRiemannianInnerSDS}
  B_0=\Omega\oplus 1\oplus 1,
\end{equation}
then $qS_\sub(\Box_1)q^{-1}$ has imaginary part (with respect to $B_0$) of size $\cO(\eps)$. Put differently, $S_\sub(\Box_1)$ has imaginary part of size $\cO(\eps)$ relative to the Hermitian inner product $b:=B_0(q\cdot,q\cdot)$, which is the symbol of a pseudodifferential inner product on $\pi^*T^*M$. We can now invoke Lemma~\ref{LemmaNilpotentInheritance} on a neighborhood of $\Gamma\cap\{|\sigma|=1\}$ and use the homogeneity of $q,b$ and $S_\sub(\Box_1)$ to obtain:

\begin{thm}
\label{ThmSDSNilpotentAtTrapping}
  For any $\eps>0$, there exists a (positive definite) $t_*$-independent pseudodifferential inner product $B=b(x,D)$ on $\cT_k M$ (thus, $b$ is an inner product on $\pi^*\cT_k M$, homogeneous of degree $0$ with respect to dilations in the base $T^*M\setminus 0$), such that
  \[
    \sup_\Gamma |\sigma|^{-1}\left\|\frac{1}{2i}(S_\sub(\Box_k)-S_\sub(\Box_k)^{*b})\right\|_b \leq \eps,
  \]
  where $\Gamma$ is the trapped set \eqref{EqTrappedSetRecall}. Put differently, there is an elliptic \psdo{}\ $Q$, invariant under $t_*$-translations, acting on sections of $\cT_k M$, with paramet\-rix $Q^-$, such that relative to the ordinary positive definite inner product~\eqref{EqRiemannianInnerSDS}, we have
  \[
    \sup_\Gamma |\sigma|^{-1}\left\|\sigma_1\left(\frac{1}{2i}(Q\Box_k Q^- - (Q\Box_k Q^-)^{*B_0})\right)\right\|_{B_0} \leq \eps.
  \]
  By restriction, the analogous statements are true for $\Box$ acting on subbundles of the tensor bundle on $M$, for instance differential forms of all degrees and symmetric 2-tensors.
\end{thm}

By the $t_*$-translation invariance of the involved symbols, inner products and operators, this is really a statement about $\Psib$-inner products, and $Q$ is a b-pseudo\-differential operator; see the discussion preceding Theorem~\ref{ThmWaveExpansion} for the relationship of the stationary and the b-picture.

\begin{rmk}
\label{RmkFirstOrderPerturbation}
  Adding a $0$-th order term to $\Box$ does not change $\Box$ or its imaginary part at the principal symbol level, thus does not affect the subprincipal operator of $\Box$ either; therefore, Theorem~\ref{ThmSDSNilpotentAtTrapping} holds in this case as well.
  
  Adding a first order operator $L$ (acting on sections of $\cT_k M$), which we assume to be $t$-independent for simplicity, does affect the subprincipal operator, more specifically its $0$-th order part, since $S_\sub(\Box+L)=S_\sub(\Box)+\sigma_1(L)$. Thus, if $\sigma_1(L)$ is small at $\Gamma$, we can use the same $\Psi$-inner product as for $\Box$ and obtain a bound on $\Im^b S_\sub(\Box+L)$ which is small, but no longer arbitrarily small. However, the bound merely needs to be smaller than $\numin/2$, see \eqref{EqBoundAtTrapping}, which does hold for small $L$.

  If we do not restrict the size of $L$, we can still obtain a spectral gap, provided one can choose a $\Psi$-inner product as in Theorem~\ref{ThmSDSNilpotentAtTrapping}, again with $\eps>0$ sufficiently (but not necessarily arbitrarily) small. This is the case if the $0$-th order part of $S_\sub(\Box+L)$ is nilpotent (or has small eigenvalues) and can be conjugated in a $t$-independent manner to an operator which is sufficiently close to being symmetric, in the sense that it satisfies the bound \eqref{EqBoundAtTrapping} with $\Box$ replaced by $\Box+L$.
\end{rmk}

We remark that the subprincipal operator $iS_\sub(\Box)=\ham_G+i\sigma_\sub(G)$ induces a notion of parallel transport on $\pi^*\cT_k M$ along the Hamilton flow of $\ham_G$. As a consequence of the nilpotent structure of $S_\sub(\Box)$ at the trapped set, parallel sections along the trapped set grow only polynomially in size (with respect to a fixed $t$-invariant positive definite inner product), rather than exponentially. Parallel sections as induced by $S_\sub(\Box+L)$, with $L$ as in Remark~\ref{RmkFirstOrderPerturbation}, may grow exponentially, with their size bounded by $C e^{\kappa|\sigma|t}$ for some constants $C>0$ and $\kappa$, where the additional factor of $|\sigma|$ in the exponent accounts for the homogeneity of the parallel transport. If such a bound does not hold for any $\kappa<\numin/2$, the dispersion of waves concentrated at the trapped set caused by the normally hyperbolic nature of the trapping is expected to be too weak to counteract the exponential growth caused by the subprincipal part of $\Box+L$, and correspondingly one does not expect a spectral gap. Notice that the growth of parallel sections is an averaged condition in that it involves the behavior of the parallel transport for large times, while the choice of $\Psi$-inner products as explained above is a local condition and depends on the pointwise structure of $S_\sub(\Box)$; thus, establishing spectral gaps only using averaged data is an interesting natural problem, even in the scalar setting.


\end{document}